\documentclass[11pt,letterpaper]{amsart}
\usepackage[utf8]{inputenc}
\usepackage{kantlipsum} 
\usepackage{comment}
\calclayout
\newif\ifdraft
\draftfalse
\usepackage{amsmath,amsfonts,amsthm,mathrsfs}
\usepackage{amssymb, stmaryrd}

\usepackage[unicode,bookmarks,pagebackref]{hyperref}
\usepackage[dvipsnames]{xcolor}
\hypersetup{colorlinks=true,citecolor=NavyBlue,linkcolor=Brown,urlcolor=Orange}

\usepackage[alphabetic]{amsrefs}

\usepackage{enumitem}

\usepackage{chngcntr}

\ifdraft
\usepackage[notcite,notref,color]{showkeys}

\definecolor{labelkey}{gray}{0.5}
\fi

\usepackage{tikz}

\usepackage{tikz-cd}

\usetikzlibrary{matrix,arrows}
\newlength{\myarrowsize} 

\pgfarrowsdeclare{cmto}{cmto}{
	\pgfsetdash{}{0pt} 
	\pgfsetbeveljoin 
	\pgfsetroundcap 
	\setlength{\myarrowsize}{0.6pt}
	\addtolength{\myarrowsize}{.5\pgflinewidth}
	\pgfarrowsleftextend{-4\myarrowsize-.5\pgflinewidth} 
	\pgfarrowsrightextend{.8\pgflinewidth}
}{
	\setlength{\myarrowsize}{0.6pt} 
  	\addtolength{\myarrowsize}{.5\pgflinewidth}  
	\pgfsetlinewidth{0.5\pgflinewidth}
	\pgfsetroundjoin
	\pgfpathmoveto{\pgfpoint{1.5\pgflinewidth}{0}}
	\pgfpatharc{-109}{-170}{4\myarrowsize}
	\pgfpatharc{10}{189}{0.58\pgflinewidth and 0.2\pgflinewidth}
	\pgfpatharc{-170}{-115}{4\myarrowsize+\pgflinewidth}
	\pgfpathclose
	\pgfusepathqfillstroke
	\pgfpathmoveto{\pgfpoint{1.5\pgflinewidth}{0}}
	\pgfpatharc{109}{170}{4\myarrowsize}
	\pgfpatharc{-10}{-189}{0.58\pgflinewidth and 0.2\pgflinewidth}
	\pgfpatharc{170}{115}{4\myarrowsize+\pgflinewidth}
	\pgfpathclose
	\pgfusepathqfillstroke
	\pgfsetlinewidth{2\pgflinewidth}
}

\pgfarrowsdeclare{cmonto}{cmonto}{
	\pgfsetdash{}{0pt} 
	\pgfsetbeveljoin 
	\pgfsetroundcap 
	\setlength{\myarrowsize}{0.6pt}
	\addtolength{\myarrowsize}{.5\pgflinewidth}
	\pgfarrowsleftextend{-4\myarrowsize-.5\pgflinewidth} 
	\pgfarrowsrightextend{.8\pgflinewidth}
}{
	\setlength{\myarrowsize}{0.6pt} 
  	\addtolength{\myarrowsize}{.5\pgflinewidth}  
	\pgfsetlinewidth{0.5\pgflinewidth}
	\pgfsetroundjoin
	\pgfpathmoveto{\pgfpoint{1.5\pgflinewidth}{0}}
	\pgfpatharc{-109}{-170}{4\myarrowsize}
	\pgfpatharc{10}{189}{0.58\pgflinewidth and 0.2\pgflinewidth}
	\pgfpatharc{-170}{-115}{4\myarrowsize+\pgflinewidth}
	\pgfpathclose
	\pgfusepathqfillstroke
	\pgfpathmoveto{\pgfpoint{1.5\pgflinewidth}{0}}
	\pgfpatharc{109}{170}{4\myarrowsize}
	\pgfpatharc{-10}{-189}{0.58\pgflinewidth and 0.2\pgflinewidth}
	\pgfpatharc{170}{115}{4\myarrowsize+\pgflinewidth}
	\pgfpathclose
	\pgfusepathqfillstroke
	\pgfpathmoveto{\pgfpoint{1.5\pgflinewidth-0.3em}{0}}
	\pgfpatharc{-109}{-170}{4\myarrowsize}
	\pgfpatharc{10}{189}{0.58\pgflinewidth and 0.2\pgflinewidth}
	\pgfpatharc{-170}{-115}{4\myarrowsize+\pgflinewidth}
	\pgfpathclose
	\pgfusepathqfillstroke
	\pgfpathmoveto{\pgfpoint{1.5\pgflinewidth-0.3em}{0}}
	\pgfpatharc{109}{170}{4\myarrowsize}
	\pgfpatharc{-10}{-189}{0.58\pgflinewidth and 0.2\pgflinewidth}
	\pgfpatharc{170}{115}{4\myarrowsize+\pgflinewidth}
	\pgfpathclose
	\pgfusepathqfillstroke
	\pgfsetlinewidth{2\pgflinewidth}
}

\pgfarrowsdeclare{cmhook}{cmhook}{
	\pgfsetdash{}{0pt} 
	\pgfsetbeveljoin 
	\pgfsetroundcap 
	\setlength{\myarrowsize}{0.6pt}
	\addtolength{\myarrowsize}{.5\pgflinewidth}
	\pgfarrowsleftextend{-4\myarrowsize-.5\pgflinewidth} 
	\pgfarrowsrightextend{.8\pgflinewidth}
}{
	\setlength{\myarrowsize}{0.6pt} 
  	\addtolength{\myarrowsize}{.5\pgflinewidth}  
 	\pgfsetdash{}{0pt}
	\pgfsetroundcap
	\pgfpathmoveto{\pgfqpoint{0pt}{-4.667\pgflinewidth}}
	\pgfpathcurveto
    {\pgfqpoint{4\pgflinewidth}{-4.667\pgflinewidth}}
    {\pgfqpoint{4\pgflinewidth}{0pt}}
    {\pgfpointorigin}
	\pgfusepathqstroke
}

\newenvironment{diagram*}[2]{%
\[%
\begin{tikzpicture}[>=cmto,baseline=(current bounding box.center),%
	to/.style={->,font=\scriptsize,cap=round},%
	into/.style={cmhook->,font=\scriptsize,cap=round},%
	onto/.style={-cmonto,font=\scriptsize,cap=round},%
	math/.style={matrix of math nodes, row sep=#2, column sep=#1,%
		text height=1.5ex, text depth=0.25ex}]%
}{%
\end{tikzpicture}%
\]%
\ignorespacesafterend%
}

\newcommand{\Dmod}{\mathscr{D}}

\newcommand{\cohH}{\mathcal{H}}

\newcommand{\NN}{\mathbb{N}}
\newcommand{\ZZ}{\mathbb{Z}}
\newcommand{\QQ}{\mathbb{Q}}

\newcommand{\CC}{\mathbb{C}}

\DeclareMathOperator{\Gr}{Gr}
\DeclareMathOperator{\Spe}{Sp}
\DeclareMathOperator{\HRH}{HRH}
\DeclareMathOperator{\DR}{DR}

\newcommand{\shf}[1]{\mathscr{#1}}
\newcommand{\OX}{\shf{O}_X}

\def\overbar#1#2#3{{%
	\setbox0=\hbox{\displaystyle{#1}}%
	\dimen0=\wd0
	\advance\dimen0 by -#2 
	\vbox {\nointerlineskip \moveright #3 \vbox{\hrule height 0.3pt width \dimen0}%
		\nointerlineskip \vskip 1.5pt \box0}%
}}

\newcommand{\shO}{\shf{O}}

\newcommand{\cK}{\mathcal{K}}
\newcommand{\cL}{\mathcal{L}}
\newcommand{\cM}{\mathcal{M}}
\newcommand{\cN}{\mathcal{N}}
\newcommand{\cQ}{\mathcal{Q}}
\newcommand{\de}{\partial}

\makeatletter
\let\@@seccntformat\@seccntformat
\renewcommand*{\@seccntformat}[1]{%
  \expandafter\ifx\csname @seccntformat@#1\endcsname\relax
    \expandafter\@@seccntformat
  \else
    \expandafter
      \csname @seccntformat@#1\expandafter\endcsname
  \fi
    {#1}%
}
\newcommand*{\@seccntformat@subsection}[1]{%
  \textbf{\csname the#1\endcsname.}
}
\makeatother

\makeatletter
\let\@paragraph\paragraph
\renewcommand*{\paragraph}[1]{%
	\vspace{0.3\baselineskip}%
	\@paragraph{\textit{#1}}%
}
\makeatother

\counterwithin{equation}{subsection}
\counterwithout{subsection}{section}
\counterwithin{figure}{subsection}

\newtheorem{theorem}[equation]{Theorem}
\newtheorem*{theorem*}{Theorem}
\newtheorem{lemma}[equation]{Lemma}
\newtheorem*{lemma*}{Lemma}
\newtheorem{corollary}[equation]{Corollary}
\newtheorem{proposition}[equation]{Proposition}

\newtheorem*{proposition*}{Proposition}
\newtheorem{conjecture}[equation]{Conjecture}

\theoremstyle{definition}

\newtheorem*{definition*}{Definition}
\newtheorem{remark}[equation]{Remark}

\newtheorem{example}[equation]{Example}
\newtheorem*{example*}{Example}
\newtheorem*{problem*}{Problem}

\theoremstyle{plain}

\newcommand{\theoremref}[1]{\hyperref[#1]{Theorem~\ref*{#1}}}
\newcommand{\lemmaref}[1]{\hyperref[#1]{Lemma~\ref*{#1}}}
\newcommand{\definitionref}[1]{\hyperref[#1]{Definition~\ref*{#1}}}
\newcommand{\propositionref}[1]{\hyperref[#1]{Proposition~\ref*{#1}}}
\newcommand{\conjectureref}[1]{\hyperref[#1]{Conjecture~\ref*{#1}}}
\newcommand{\corollaryref}[1]{\hyperref[#1]{Corollary~\ref*{#1}}}
\newcommand{\exampleref}[1]{\hyperref[#1]{Example~\ref*{#1}}}
\newcommand{\setupref}[1]{\hyperref[#1]{Set-up~\ref*{#1}}}
\newcommand{\remarkref}[1]{\hyperref[#1]{Remark~\ref*{#1}}}
\newcommand{\claimref}[1]{\hyperref[#1]{Claim~\ref*{#1}}}
\newcommand{\figureref}[1]{\hyperref[#1]{Figure~\ref*{#1}}}

\makeatletter
\let\old@caption\caption
\renewcommand*{\caption}[1]{%
	\setcounter{figure}{\value{equation}}%
	\stepcounter{equation}%
	\old@caption{#1}\relax%
}
\makeatother

\newcounter{intro}

\newtheorem{intro-conjecture}[intro]{Conjecture}
\newtheorem{intro-corollary}[intro]{Corollary}
\newtheorem{intro-theorem}[intro]{Theorem}

\def\cB{\mathcal{B}}

\def\cK{\mathcal{K}}

\newcommand{\parref}[1]{\hyperref[#1]{\S\ref*{#1}}}

\makeatletter
\newcommand*\if@single[3]{%
  \setbox0\hbox{${\mathaccent"0362{#1}}^H$}%
  \setbox2\hbox{${\mathaccent"0362{\kern0pt#1}}^H$}%
  \ifdim\ht0=\ht2 #3\else #2\fi
  }
\newcommand*\rel@kern[1]{\kern#1\dimexpr\macc@kerna}
\newcommand*\widebar[1]{\@ifnextchar^{{\wide@bar{#1}{0}}}{\wide@bar{#1}{1}}}
\newcommand*\wide@bar[2]{\if@single{#1}{\wide@bar@{#1}{#2}{1}}{\wide@bar@{#1}{#2}{2}}}
\newcommand*\wide@bar@[3]{%
  \begingroup
  \def\mathaccent##1##2{%
    \if#32 \let\macc@nucleus\first@char \fi
    \setbox\z@\hbox{$\macc@style{\macc@nucleus}_{}$}%
    \setbox\tw@\hbox{$\macc@style{\macc@nucleus}{}_{}$}%
    \dimen@\wd\tw@
    \advance\dimen@-\wd\z@
    \divide\dimen@ 3
    \@tempdima\wd\tw@
    \advance\@tempdima-\scriptspace
    \divide\@tempdima 10
    \advance\dimen@-\@tempdima
    \ifdim\dimen@>\z@ \dimen@0pt\fi
    \rel@kern{0.6}\kern-\dimen@
    \if#31
      \overline{\rel@kern{-0.6}\kern\dimen@\macc@nucleus\rel@kern{0.4}\kern\dimen@}%
      \advance\dimen@0.4\dimexpr\macc@kerna
      \let\final@kern#2%
      \ifdim\dimen@<\z@ \let\final@kern1\fi
      \if\final@kern1 \kern-\dimen@\fi
    \else
      \overline{\rel@kern{-0.6}\kern\dimen@#1}%
    \fi
  }%
  \macc@depth\@ne
  \let\math@bgroup\@empty \let\math@egroup\macc@set@skewchar
  \mathsurround\z@ \frozen@everymath{\mathgroup\macc@group\relax}%
  \macc@set@skewchar\relax
  \let\mathaccentV\macc@nested@a
  \if#31
    \macc@nested@a\relax111{#1}%
  \else
    \def\gobble@till@marker##1\endmarker{}%
    \futurelet\first@char\gobble@till@marker#1\endmarker
    \ifcat\noexpand\first@char A\else
      \def\first@char{}%
    \fi
    \macc@nested@a\relax111{\first@char}%
  \fi
  \endgroup
}
\makeatother

\usepackage{xpatch}
\makeatletter   
\xpatchcmd{\@tocline}
{\hfil\hbox to\@pnumwidth{\@tocpagenum{#7}}\par}
{\ifnum#1<0\hfill\else\dotfill\fi\hbox to\@pnumwidth{\@tocpagenum{#7}}\par}
{}{}
\makeatother    

\makeatletter
\def\l@section{\@tocline{1}{0pt}{0pc}{0pc}{\bfseries}}
 \def\l@subsection{\@tocline{2}{0pt}{4pc}{6pc}{}}
\def\l@subsubsection{\@tocline{3}{0pt}{8pc}{8pc}{}}
 \makeatother

\begin{document}

\author[B.~Dirks]{Bradley Dirks}

\address{Department of Mathematics, Stony Brook University, Stony Brook, NY 11794-3651, USA}

\email{bradley.dirks@stonybrook.edu}

\author[S.~Olano]{Sebasti\'{a}n Olano}

\address{Department of Mathematics, University of Toronto, 40 St. George St., Toronto, Ontario Canada, M5S 2E4}

\email{seolano@math.toronto.edu}

\author[D.~Raychaudhury]{Debaditya Raychaudhury}

\address{Department of Mathematics and Statistics, University of New Mexico, Albuquerque, NM 87131, USA}

\email{rcdeba@gmail.com}

\thanks{This material is based upon work supported by the National Science Foundation under Grant No. DMS-1926686 and MSPRF DMS-2303070. The research of DR is partially supported by an AMS-Simons Travel Grant.}

\title[A Hodge Theoretic Generalization of $\mathbb{Q}$-Homology Manifolds: LCI Case]{A Hodge Theoretic Generalization of $\mathbb{Q}$-Homology Manifolds II: Local Complete Intersections}


\begin{abstract} Recently, the authors introduced and studied a singularity invariant of a complex algebraic variety $Z$, written $\HRH(Z)$ (for ``Hodge rational homology'' manifold level). In this paper, we focus on local complete intersection subvarieties. We relate $\HRH(Z)$ to various well-known invariants, like Bernstein--Sato polynomials and the Dimca-Maisonobe-Saito spectrum. In the hypersurface case it turns out that $\HRH(Z)$ can be completely characterized by these invariants, though higher codimension case is more subtle.
\end{abstract}

\vspace*{-25pt}

\maketitle




\section{Introduction}
For hypersurfaces, the condition of being a rational homology manifold admits a simple numerical characterization. Suppose $Z=V(f)$ is a hypersurface in a smooth variety $X$, and let $b_f(s)$ be the Bernstein--Sato polynomial. Then $Z$ is a rational homology manifold if and only if the only integral root of $b_f(s)$ is $-1$, with multiplicity one. Moreover, this is equivalent to the vanishing of the unipotent vanishing cycles $\varphi_{f,1}(\shO_X)$ \cite{VanCycleRHM}*{Thm. 3}.

Hodge-theoretic weakenings of the rational homology manifold condition for singular varieties \cites{DOR1, ParkPopaSymmetry} have attracted interest due to their connections with higher singularities \cites{CDM, DOR1, PSV} and failures of $\mathbb Q$-factoriality \cite{ParkPopaQFact}. In this paper, we relate the invariant $\HRH(Z)$ introduced in \cite{DOR1} to other well-known invariants in the case where $Z$ is a complete intersection. Throughout, we let $X$ be a smooth connected complex algebraic variety with $\dim X=n$.

We now fix an embedding $i \colon Z \hookrightarrow X$ as a complete intersection subvariety of pure codimension $r$. Moreover, we fix $f_1,\dots, f_r \in \shO_X(X)$ such that $Z=V(f_1,\dots, f_r)$. In this setting, there are several singularity invariants of $Z$ defined through $\Dmod$-module and mixed Hodge module theory. We will not recall all of their definitions in the introduction, but we describe the main objects needed to state our results.

We first discuss the hypersurface case, so we take for now $r=1$ and $f_1=f$. In this case, the picture can be understood using the well-known properties of the $V$-filtration, reviewed in \S \ref{subsect-VFiltSpecSpec} below.
This case also serves as the model for the invariants that appear for local complete intersections.


In the hypersurface case there are three well-established singularity invariants to consider (their local notions at $x\in Z$ being denoted by a subscript $x$ as in \theoremref{thmhypersurfaces}): 
\begin{itemize} \item $\widetilde{\alpha}_{\ZZ}(f)$, the smallest (negated) integer root of the reduced Bernstein--Sato polynomial $\widetilde{b}_f(s)$.

\item ${\rm Sp}_{\min,\ZZ}(Z,x)$, the smallest integer spectral number in the Steenbrink spectrum \cite{SteenbrinkSpectrum}.

\item $p(\varphi_{f,1}(\shO_X),F)$, the index where the first jump in the Hodge filtration on the mixed Hodge module $\varphi_{f,1}(\shO_X)$ occurs.
\end{itemize}

Conventionally, these invariants are $+\infty$ if the set of which they are the ``smallest element'' is empty. As stated above, $\widetilde{\alpha}_{\ZZ}(f) = +\infty$ if and only if $p(\varphi_{f,1}(\shO_X),F) = +\infty$ if and only if $Z$ is a rational homology manifold. Note also that up to a shift, $p(\varphi_{f,1}(\shO_X),F)$ is the invariant $\widetilde{\alpha}^{\rm min. int}(f)$ defined in \cite{JKSY}. The latter is the first integer jumping index of the microlocal $V$-filtration \cite{SaitoMicrolocal}. 

These are related to each other in the following theorem.

\begin{intro-theorem}\label{thmhypersurfaces} Assume $Z$ is defined inside an $n$-dimensional smooth variety $X$ by a regular element $f\in \shO_X(X)$. Then for any $x\in Z$, we have
\[ \widetilde{\alpha}_{\ZZ,x}(f) -2 \leq {\rm HRH}_x(Z) = p(\varphi_{f,1}(\shO_X)_x,F)+ n - 2 \leq {\rm Sp}_{\min, \ZZ}(Z,x)-2,\]
where the second inequality is an equality if $x$ is an isolated singular point.
\end{intro-theorem}

\begin{remark} In \cite{JKSY}, it was noted that in the isolated hypersurface singularities case, we have \[\widetilde{\alpha}^{\rm min. int}(f) = {\rm Sp}_{\min,\ZZ}(Z,x),\] which proves the last assertion in the theorem statement. We give another proof of this below (see \remarkref{rmk-HypersurfaceCase}).
\end{remark}

The Bernstein--Sato characterization of rational homology manifolds is specific to hypersurfaces. In higher codimension, the analogous statement is false: Torrelli's example gives a complete intersection which is a rational homology manifold, but for which the corresponding reduced Bernstein--Sato polynomial has an integral root; see \exampleref{eg-Torrelli}. To study local complete intersections, one must also replace the unipotent vanishing cycles of a hypersurface; a natural replacement is the unipotent Verdier specialization along $Z$. The results below develop the corresponding comparison picture for local complete intersections, relating the resulting specialization invariant to $\HRH$, Bernstein--Sato roots, and the integral spectrum.

We now return to the local complete intersection case. The central construction we use to study local complete intersections is the \emph{unipotent Verdier specialization} of $\shO_X$ along $Z$. This is a monodromic mixed Hodge module ${\rm Sp}_Z(\shO_X)^\ZZ$ on $X\times \mathbb A^r_z$ which sits in a short exact sequence
\[ 0 \to i_* \QQ^H_{Z\times \mathbb A^r_z}[n] \to {\rm Sp}_Z(\shO_X)^\ZZ \to Q^\ZZ \to 0,\]
where the left-hand side is the trivial Hodge module on $Z\times \mathbb A^r_z$ and the module $Q^\ZZ$ is defined through this short exact sequence, though it was observed in \cite{DirksMicrolocal} that it is related to another important construction we will see below. 

 The first integer invariant we consider is $p(Q^\ZZ,F)$, where $F$ is the Hodge filtration on the underlying $\Dmod$-module of $Q$. In the hypersurface case, we have $p(Q^\ZZ,F) = p(\varphi_{f,1}(\shO_X),F) -1$. 

The module $Q^\ZZ$ is related to the integral spectrum of $Z$ at $x\in Z$, defined by \cite{DMS}. As in the hypersurface case, let ${\rm Sp}_{\rm min,\ZZ}(Z,x)$ denote the smallest non-zero element of the integral spectrum of $Z$ at $x$. 

The main result about these invariants is the following and is proved in \propositionref{prop-InequalityPRS} and \propositionref{prop-SpectrumBound}.

\begin{intro-theorem} \label{thm-main} Let $x\in Z$ be a point in a local complete intersection subvariety of the smooth variety $X$. Then we have the following inequalities:
\begin{enumerate} 
\item $ p(Q^\ZZ,F) + n -1 \leq \HRH(Z)$.
\item $p(Q_x^\ZZ,F) + n +1 \leq {\rm Sp}_{\min,\ZZ}(Z,x)$.
\end{enumerate}
\end{intro-theorem}

In the isolated hypersurface case, the integer spectrum completely determines ${\rm HRH}_x(Z)$. This is no longer the case in the isolated LCI case, but we still have the following inequality.

\begin{intro-theorem}\label{thmspectrum}
    Let $Z$ have an isolated local complete intersection singularity at $x\in Z$. Then \[ \HRH_x(Z) \geq \Spe_{\min,\ZZ}(Z,x) - 2.\]
\end{intro-theorem}

It is clear from \theoremref{thm-main} that $Q^\ZZ = 0$ implies $Z$ is a rational homology manifold and that ${\rm Sp}_{\min,\ZZ}(Z,x) = + \infty$ for all $x\in Z$. However, this condition is too strong to be equivalent to the rational homology manifold property of $Z$, as we see in \exampleref{eg-Torrelli} below. In any case, we can characterize the vanishing $Q^{\ZZ} = 0$  as we explain now.

In \cites{Mustata,CDMO}, singularities of $f_1,\dots, f_r \in \shO_X(X)$ are related to those of the \emph{general linear combination hypersurface} $g = \sum_{i=1}^r y_i f_i$ defined on $Y = X\times \mathbb A^r_y$. Let $U = Y \setminus (X \times \{0\})$. The hypersurface defined by $g\vert_U$ is used in \cites{CDMO,CDM, DirksMicrolocal} to study higher singularities of $Z$. We have the following characterization of the vanishing of $Q^\ZZ$ in terms of the rational homology manifold property of $V(g\vert_U) \subseteq U$. We use the notation $\HRH(V(g\vert_U)) = \HRH(g\vert_U)$.

\begin{intro-theorem}\label{thm-openU} In the notation above, we have
\[ Q^\ZZ = 0 \iff V(g\vert_U) \text{ is a rational homology manifold} \iff \varphi_{g\vert_U,1}(\shO_U) = 0.\]
Moreover, we have inequalities
\begin{enumerate}
    \item $\widetilde{\alpha}_\ZZ(g\vert_U) \leq p(Q^\ZZ,F) + n + r$,
    \item $\widetilde{\alpha}_{\ZZ}(g\vert_U) -r -1 \leq \HRH(Z)$,
    \item $\HRH(g\vert_U) - r +1 \leq \HRH(Z)$.
\end{enumerate}
\end{intro-theorem}



In fact, the results above follow from more precise statements tracking how the Hodge filtrations on the different monodromic pieces of $Q^\ZZ$ relate to one another, measured by certain ``missed jumps'' between them (see \remarkref{rmk-jumpsInterpretation}). This also leads to criteria for when the inequalities above are equalities (see \propositionref{prop-equalityPRS}, \corollaryref{cor-HRHU}, and \corollaryref{corprsbound}). \exampleref{ex-jumpsliminal} shows that these missed jumps can behave in quite different ways. It would be interesting to understand whether they admit a more systematic interpretation.

The tuple $f_1,\dots, f_r$ has a \emph{Bernstein--Sato polynomial} $b_f(s)$ \cite{BMS} which is divisible by $(s+r)$ in this case, and so we can consider the \emph{reduced Bernstein--Sato polynomial} $\widetilde{b}_f(s) = b_f(s)/(s+r)$. Define
\[ \widetilde{\alpha}_\ZZ(Z) = \min\{j \in \ZZ \mid \widetilde{b}_f(-j) = 0\},\]
and conventionally set $\widetilde{\alpha}_{\ZZ}(Z) = +\infty$ if there are no integer roots of $\widetilde{b}_f(s)$. We can also consider the local notions for $x\in Z$, denoted with a subscript $x$.

To relate this invariant to the ones we have already discussed, we need to assume that $Z$ has rational singularities.

\begin{intro-corollary}\label{cor-bs}
If $Z$ has rational singularities, then
\begin{enumerate} 
\item $\widetilde{\alpha}_{\ZZ}(Z) \leq p(Q^\ZZ,F) + n +r,$ 
\item $\widetilde{\alpha}_{\ZZ}(Z)-r - 1 \leq \HRH(Z),$
\item $\widetilde{\alpha}_{\ZZ,x}(Z) -r +1 \leq {\rm Sp}_{\min,\ZZ}(Z,x).$
\end{enumerate}

In particular, if $Z$ has rational singularities, then $\widetilde{\alpha}_\ZZ(Z) = + \infty$ implies $Z$ is a rational homology manifold.
\end{intro-corollary}

\begin{remark} We pose Conjecture \ref{conj-BFunction} on the structure of certain Bernstein--Sato polynomials. The validity of that conjecture would illuminate the connection between $\widetilde{\alpha}_{\ZZ}(Z)$ and $Q^{\ZZ}$, as in \corollaryref{cor-QMinExpZ}.
\end{remark}

In fact, we show a more precise statement, but it becomes rather technical and is discussed at the end of Section \ref{sect-LCICase} below.
We remark that even in the isolated hypersurface singularities case, it is possible to have strict inequality \[\widetilde{\alpha}_{\ZZ,x}(Z) < {\rm Sp}_{\min,\ZZ}(Z,x),\]
see \cite{JKSY}*{Rmk 3.4d}. Moreover, the converse to the last statement of the corollary is not true (\exampleref{eg-Torrelli}).

\medskip

\noindent \textbf{Plan of the paper.} \textcolor{black}{Section \ref{sec-Prelim} contains a review of the V-filtration, the Specialization construction, and of the definition of the spectrum given by \cite{DMS}.}
Section \ref{sect-LCICase} studies the integer invariants $\HRH(Z)$, ${\rm Sp}_{\min,\ZZ}(Z,x)$, $\widetilde{\alpha}_\ZZ(Z)$, and $p(Q^\ZZ,F)$. The invariant $\HRH(Z)$ is characterized via the $V$-filtration and \theoremref{thmhypersurfaces} is proven in \S \ref{sectionhypersurfaces}. We prove \theoremref{thm-main} ($=$ \propositionref{prop-InequalityPRS} and \propositionref{prop-SpectrumBound}), \theoremref{thm-openU}, and \corollaryref{cor-bs} in \S \ref{sec-intinv}. The proof of \theoremref{thmspectrum} is contained in \S \ref{sec-iso}. \textcolor{black}{Examples of various features are discussed in Section \ref{sec-ex}.}

\medskip

\noindent {\bf Acknowledgments.} We would like to thank Bhargav Bhatt, Qianyu Chen, Radu Laza, Lauren\c{t}iu Maxim, Mircea Musta\c{t}\u{a}, Sung Gi Park, Mike Perlman, Mihnea Popa, Sridhar Venkatesh and Anh Duc Vo for many conversations on the topics in this paper.

\section{Preliminaries} \label{sec-Prelim}
\subsection{V-filtration} \label{subsect-VFiltSpecSpec}
To define nearby and vanishing cycles for mixed Hodge modules, Saito uses the $V$-filtration of Kashiwara and Malgrange. As this is arguably the most important construction for what follows, we remind the reader of its definition.

For the smooth variety $X$, let $T = X\times \mathbb{A}^r_t$ be the trivial vector bundle over $X$ with fiber coordinates $t_1,\dots, t_r$. We have $\Dmod_T = \Dmod_X\langle t_1,\dots ,t_r,\de_{t_1},\dots ,\de_{t_r}\rangle$, where as usual $[\de_{t_i},t_j] = \delta_{ij}$, the Kronecker delta. This ring carries a $\ZZ$-indexed, decreasing filtration
\[ V^\bullet \Dmod_T = \left\{ \sum_{\beta,\gamma} P_{\beta,\gamma} t^\beta \de_t^\gamma \mid P_{\beta,\gamma} \in \Dmod_X,\, |\beta| \geq |\gamma| + \bullet\right\},\]
so that, for example, $t_i \in V^1 \Dmod_T$, $\de_{t_j} \in V^{-1} \Dmod_T$, and
\[ V^j\Dmod_T \cdot V^k \Dmod_T \subseteq V^{j+k}\Dmod_T.\]

If $\cM$ is a regular holonomic $\Dmod_T$-module underlying a mixed Hodge module, then it admits a $\QQ$-indexed $V$-filtration along $(t_1,\dots, t_r)$. This is the unique exhaustive, decreasing, $\QQ$-indexed filtration $(V^\alpha \cM)_{\alpha \in \QQ}$ which is discrete\footnote{Meaning there exists $\{\alpha_j\}_{j\in \ZZ}$ with $\lim_{j\to \pm \infty} \alpha_j =  \pm \infty$ so that $V^\alpha \cM$ is constant for $\alpha \in (\alpha_j,\alpha_{j+1})$.} and left continuous\footnote{Meaning $V^\alpha \cM = \bigcap_{\beta < \alpha} V^\beta \cM$.}, and which satisfies the following properties:
\begin{enumerate} \item For any $\alpha \in \QQ, j\in \ZZ$, we have containment $V^j\Dmod_T \cdot V^\alpha \cM \subseteq V^{\alpha+j}\cM$.

\item For $\alpha \gg 0, j \in \ZZ_{\geq 0}$, we have equality $V^j \Dmod_T \cdot V^\alpha \cM = V^{\alpha+j}\cM$.

\item For all $\alpha \in \QQ$, the $V^0\Dmod_T$-module $V^\alpha \cM$ is coherent.

\item For $s = -\sum_{i=1}^r \de_{t_i} t_i$ and for any $\alpha \in \QQ$, there exists some $N \gg 0$ such that 
\[(s+\alpha)^N V^\alpha \cM \subseteq \bigcup_{\beta > \alpha} V^\beta \cM = V^{>\alpha}\cM.\] In other words, $s+\alpha$ is nilpotent on ${\rm Gr}_V^\alpha(\cM) = V^\alpha \cM/ V^{>\alpha}\cM$.
\end{enumerate}

\begin{example} If $\sigma \colon X \to T$ is the inclusion of the zero section, then for any mixed Hodge module $N$ on $X$, the push forward $\sigma_* N$ underlies a mixed Hodge module on $T$. Its $\Dmod_T$-module can be written
\[ \sigma_+ \cN = \bigoplus_{\alpha \in \NN^r} \cN \de_t^\alpha\delta_0,\]
where $\delta_0$ is a formal symbol which is annihilated by $t_1,\dots,t_r$. Then
\[V^{\lambda}\sigma_+\cN = \bigoplus_{|\alpha|=0}^{\lfloor -\lambda\rfloor} \cN\de_t^\alpha \delta_0.\]
\end{example}

For $(\cM,F)$ a filtered $\Dmod_T$-module underlying a mixed Hodge module on $T$, we define Koszul-like complexes
\[ F_p A^\chi(\cM) = \left[F_p V^\chi \cM \xrightarrow[]{t} \bigoplus_{i=1}^r F_p V^{\chi+1} \cM \xrightarrow[]{t} \dots \xrightarrow[]{t} F_p V^{\chi+r} \cM\right],\]
\[ F_p B^\chi(\cM) = F_p A^\chi(\cM)/F_p A^{>\chi}(\cM) = \left[F_p {\rm Gr}_V^\chi(\cM) \xrightarrow[]{t} \bigoplus_{i=1}^r F_p {\rm Gr}_V^{\chi+1}(\cM) \xrightarrow[]{t} \dots \xrightarrow[]{t} F_p {\rm Gr}_V^{\chi+r}(\cM)\right],\]
\[ F_p C^\chi(\cM) = \left[F_{p-r}{\rm Gr}_V^{\chi+r}(\cM) \xrightarrow[]{\de_t} \bigoplus_{i=1}^r F_{p-r+1}{\rm Gr}_V^{\chi+r-1}(\cM) \xrightarrow[]{\de_t} \dots \xrightarrow[]{\de_t} F_p {\rm Gr}_V^{\chi} (\cM)\right].\]

The last condition in the definition of the $V$-filtration leads to the following acyclicity results for these complexes.

\begin{proposition}[\cite{CD}*{Thm. 3.1, 3.2}] \label{prop-KoszulAcyclic} For all $\chi\neq 0$, the complexes $B^\chi(\cM), C^\chi(\cM)$ are acyclic. For $\chi > 0$, the complex $A^\chi(\cM)$ is acyclic.
\end{proposition}

When $r=1$, this acyclicity means we have isomorphisms
\[ t \colon V^\alpha \cM \cong V^{\alpha+1}\cM \text{ for } \alpha > 0\]
\[ t\colon {\rm Gr}_V^\alpha(\cM) \cong {\rm Gr}_V^{\alpha+1}(\cM) \text{ for } \alpha \neq 0\]
\[ \de_t \colon {\rm Gr}_V^{\alpha+1}(\cM) \cong {\rm Gr}_V^{\alpha}(\cM) \text{ for } \alpha \neq 0.\]

For $r=1$, one of the properties which the filtered module $(\cM,F)$ must satisfy to underlie a mixed Hodge module on $T$ is that these isomorphisms are \emph{filtered isomorphisms} in certain ranges: specifically, Saito imposes that
\[ t \colon F_p V^\alpha \cM \cong F_p V^{\alpha+1}\cM \text{ for } \alpha > 0\]
\[ \de_t \colon F_p {\rm Gr}_V^{\alpha+1}(\cM) \cong F_{p+1}{\rm Gr}_V^{\alpha}(\cM) \text{ for } \alpha < 0.\]

Some immediate consequences of these conditions are the following:
\begin{proposition}[\cite{SaitoMHP}*{Sect. 3}]\label{prop-VFilt} Let $(\cM,F)$ underlie a mixed Hodge module on $T = X\times \mathbb A^1_t$. Let $j\colon X\times \mathbb G_m \to X\times \mathbb A^1_t$ be the inclusion of the complement of the zero section. Then
\begin{enumerate} \item For any $\lambda > 0$ and $p \in \ZZ$, we have
\begin{equation} \label{eq-HodgePieceV} F_p V^\lambda \cM = V^\lambda \cM \cap j_*(j^*(F_p \cM)).\end{equation}

\item For all $p\in \ZZ$, we have
\[ F_p \cM = \sum_{i\geq 0} \de_t^i(F_{p-i}V^0\cM).\]
\end{enumerate}

If $\cM$ has no sub-module supported on $\{t=0\}$, then the equality \eqref{eq-HodgePieceV} holds with $\lambda = 0$. In this case, we have
\[ F_p \cM = \sum_{i\geq 0} \de_t^i(V^0\cM \cap j_*j^*(F_{p-i}\cM)).\]
\end{proposition}

For $r>1$, filtered acyclicity still holds in the corresponding ranges.

\begin{proposition}[\cites{CD,CDS}] \label{prop-KoszulFiltAcyclic} For all $\chi< 0$, the complex $C^\chi(\cM,F)$ is filtered acyclic. For $\chi > 0$, the complexes $A^\chi(\cM,F)$ and $B^\chi(\cM,F)$ are filtered acyclic.
\end{proposition}

The complexes $B^0(\cM), C^0(\cM)$ are related to the restriction functors for mixed Hodge modules.

\begin{proposition}[\cite{CD}*{Thm. 1.2} and \cite{CDS}*{Thm. 1}] \label{prop-Restriction} Let $M \in {\rm MHM}(T)$. Then $B^0(\cM,F), C^0(\cM,F)$ are strictly filtered complexes. 

 Let $\sigma \colon X \to X\times \mathbb{A}^r_t$ be the zero section. We have filtered isomorphisms
\[ F_p \cohH^j \sigma^!(\cM) \cong F_p \cohH^j B^0(\cM) \text{ for all } j\in [0,r],\]
\[ F_p \cohH^j \sigma^*(\cM) \cong F_p \cohH^j C^0(\cM) \text{ for all } j\in [-r,\dots, 0].\]
\end{proposition}

\subsection{Specialization and Spectrum}
The results of this paper make use of the Verdier specialization functor, which we review here. For details, see \cites{SaitoMHM,BMS,CD}. 

Let $Z\subseteq X$ be a (possibly singular) subvariety defined by $f_1,\dots, f_r \in \shO_X(X)$. This defines a graph embedding $\Gamma \colon X \to T$ by $x \mapsto (x,f_1(x),\dots, f_r(x))$. Given $M$ a mixed Hodge module on $X$, we obtain $\Gamma_* M$ a mixed Hodge module on $T$.

The \emph{Verdier specialization} of $M$ along $Z$ (or, of $\Gamma_* M$ along $X\times \{0\}$) is a mixed Hodge module on $X\times \mathbb A^r_z$, where $z_1,\dots, z_r$ are the fiber coordinates. The module is denoted ${\rm Sp}_Z(M) = {\rm Sp}(\Gamma_*(M))$. Its underlying filtered $\Dmod$-module is 
\[ F_p {\rm Sp}(\Gamma_*(\cM)) = \bigoplus_{\chi \in \QQ} F_p {\rm Gr}_V^{\chi}(\Gamma_*(\cM)),\]
where $V^\bullet\Gamma_*(\cM)$ is the $V$-filtration along $t_1,\dots, t_r$ and $F_p {\rm Gr}_V^\chi(\Gamma_*(\cM)) = \frac{F_p V^\chi \Gamma_*(\cM)}{F_p V^{>\chi}\Gamma_* (\cM)}$. The $\Dmod$-module action is given on $\overline{m} \in {\rm Gr}_V^\chi(\Gamma_*(\cM))$ by
\[P \overline{m} = \overline{Pm}, \text{ for } P \in \Dmod_X,\]
\[ z_i \overline{m} = \overline{t_i m} \in {\rm Gr}_V^{\chi+1}(\Gamma_*(\cM)),\]
\[ \de_{z_i}\overline{m} = \overline{\de_{t_i} m} \in {\rm Gr}_V^{\chi-1}(\Gamma_*(\cM)).\]

This gives an example of a \emph{monodromic mixed Hodge module} on $X\times \mathbb A^r_z$. Recall that a mixed Hodge module is monodromic if its underlying $\Dmod$-module is, which means that it decomposes into generalized eigenspaces for the Euler operator $\theta_z = \sum_{i=1}^r z_i \de_{z_i}$. If $\cN$ is monodromic, for any $\chi \in \QQ$, we let
\[ \cN^\chi = \bigcup_{j\geq 1} \ker((\theta_z - \chi +r)^j), \text{ so that } \cN = \bigoplus_{\chi \in \QQ} \cN^\chi.\]

The $V$-filtration along $z_1,\dots, z_r$ is particularly easy to understand for monodromic modules: indeed, it is given by
\[ V^\lambda \cN = \bigoplus_{\chi\geq \lambda} \cN^\chi,\quad {\rm Gr}_V^\lambda(\cN) \cong \cN^\lambda.\]

We have particular interest in the case $M = \QQ^H_X[n]$. We let $\Gamma_*(\QQ^H_X[n]) = B_f$ in this case for ease of notation. For $Z \subseteq X$ a complete intersection (meaning $f_1, \dots ,f_r \in \shO_X(X)$ form a regular sequence), the module ${\rm Sp}(B_f)$ admits a morphism $L \to {\rm Sp}(B_f)$, where $L = i_* \QQ^H_{Z\times \mathbb A^r_z}[n]$ is the trivial Hodge module on $i \colon Z\times \mathbb A^r_z \hookrightarrow X\times \mathbb A^r_z$.

\begin{lemma} Let \[(\cK,F)= \ker(({\rm Gr}_V^r(B_f),F[r]) \xrightarrow[]{\de_{t_i}} \bigoplus_{i=1}^r ({\rm Gr}_V^{r-1}(B_f),F[r-1])).\] Then $(\cK,F)$ underlies $i_*\QQ_Z^H[n-r]$.

Moreover, the $\Dmod$-module $\cL = \cK\boxtimes \shO_{\mathbb A^r_z} = \cK[z_1,\dots, z_r]$ underlies $L$, with filtration given by
\[F_p \cL = (F_{p+r} \cK)[z_1,\dots, z_r].\]
\end{lemma}
\begin{proof} The first claim follows by applying \propositionref{prop-Restriction} to $B_f = \Gamma_*(\QQ_X^H[n])$, using Base Change \cite{SaitoMHM}*{(4.4.3)} to see that
\[ \sigma^* \Gamma_*(\QQ_X^H[n]) = i_* i^* \QQ_X^H[n] = i_* \QQ_Z^H[n].\]

The second claim follows by definition, using, for example, \cite{SaitoMHM}*{Lem. 2.25}.
\end{proof}

The map $L \to {\rm Sp}(B_f)$ is injective, and we let $Q$ be the cokernel of the map. Then we have a short exact sequence of monodromic mixed Hodge modules on $X\times \mathbb A^r_z$:
\[ 0 \to L \to {\rm Sp}(B_f) \to Q \to 0.\]

The monodromic pieces of $\cL$ satisfy $\cL^{\chi} = 0$ for $\chi \notin \ZZ_{\geq r}$. Thus, we see that ${\rm Sp}(B_f)^\chi = \cQ^\chi$ for all $\chi \notin \ZZ_{\geq r}$. The interesting part of this short exact sequence is then
\begin{equation} \label{eq-QSES} 0 \to L \to {\rm Sp}(B_f)^{\ZZ} \to Q^{\ZZ} \to 0\end{equation}
where for any monodromic mixed Hodge module, we use the superscript $\ZZ$ to denote the ``unipotent part'', which is the direct sum of the monodromic pieces with integer indices. We also use a superscript $\alpha + \ZZ$ to denote the filtered direct summand of a monodromic module which is obtained by collecting all summands with indices in $\alpha + \ZZ$.

For $Z\subseteq X$ a local complete intersection subvariety, we review the definition of the spectrum of $Z$ at $x\in Z_{\rm sing}$ due to Dimca, Maisonobe and Saito. For details, consult \cite{DMS} and \cite{DirksMicrolocal}.

Using the monodromy endomorphism and the mixed Hodge structure on the cohomology of the Milnor fiber, Steenbrink \cite{SteenbrinkSpectrum} defined, in the isolated hypersurface singularities case, the \emph{spectrum} of the hypersurface singularity. This is an invariant of the singularity given by a multiset of positive rational numbers, encoding the eigenspaces of the monodromy operator.

In \cite{DMS}, Dimca, Maisonobe and Saito defined the spectrum for any variety $Z$ at a point $x\in Z$ using the theory of mixed Hodge modules. The definition is rather technical, but when we assume $Z$ is a local complete intersection variety it is slightly simpler. 

Let $Z\subseteq X$ be a local complete intersection subvariety of pure codimension $r$. Let $x\in Z$, and locally around $x$ we can write $Z = V(f_1,\dots, f_r)$ where $f_1,\dots, f_r \in \shO_X(X)$ form a regular sequence. Let $\xi \in \{x\} \times \mathbb A^r_z$ be a sufficiently general element, and set $i_{\xi}\colon\left\{\xi\right\}\to X\times \mathbb{A}_z^r$ to be the inclusion. Then define the \emph{non-reduced Spectrum of $Z$ at $x$} by
\[ \widehat{{\rm Sp}}(Z,x) = \sum_{\alpha \in \QQ_{>0}} m_{\alpha,x} t^\alpha,\]
where
\[ m_{\alpha,x} = \sum_{k\in \ZZ} (-1)^k \dim_{\CC} {\rm Gr}^F_{\lceil \alpha\rceil - \dim Z -1} \cohH^{k-r} i_\xi^*( {\rm Sp}(B_f)^{\alpha+\ZZ}),\]
and we define the \emph{reduced spectrum} by
\[ {\rm Sp}(Z,x) = \widehat{\rm Sp}(Z,x) + (-t)^{\dim Z +1}.\]

This definition clearly extends to an arbitrary monodromic module $M$, where we write $m_{\alpha,x}(M)$ for the alternating sum
\[ m_{\alpha,x}(M) = \sum_{k\in \ZZ} (-1)^k \dim_{\CC} {\rm Gr}^F_{\lceil \alpha\rceil - \dim Z -1} \cohH^{k-r} i_\xi^*(\cM^{\alpha+\ZZ}).\] Moreover, if 
\[0 \to M_1 \to M_2 \to M_3 \to 0\] is a short exact sequence of monodromic mixed Hodge modules, then we see that
\[ \widehat{\rm Sp}(M_2,x) = \widehat{\rm Sp}(M_1,x) + \widehat{\rm Sp}(M_3,x).\]

For $M$ a monodromic mixed Hodge module, we let its \emph{integer spectrum} be denoted
\[ \widehat{\rm Sp}_{\ZZ}(M,x) = \sum_{j \in \ZZ_{>0}} m_{j,x}(M) t^j.\]

Applying the additivity to the short exact sequence \eqref{eq-QSES} we get
\[ \widehat{\rm Sp}(Q,x) = \widehat{\rm Sp}(Z,x) - \widehat{\rm Sp}(L,x),\]
and in particular,
\[ \widehat{\rm Sp}_{\ZZ}(Q,x) = \widehat{\rm Sp}_{\ZZ}(Z,x) - \widehat{\rm Sp}(L,x),\]
which by an easy computation of the non-reduced spectrum of $L$, gives
\[ \widehat{\rm Sp}(Q,x) = {\rm Sp}(Z,x).\]

We end this subsection by stating a criterion for vanishing of integer spectral numbers, which is a special case of \cite{DirksMicrolocal}*{Lem. 2.7}:
\begin{lemma} \label{lem-lowerBoundSpectralNumbers} Let $M$ be a monodromic mixed Hodge module. Assume $F_{p-1-\dim X} \cM^{\ZZ} = 0$. Then for all $j\in \ZZ_{<p}$, we have $m_{j,x}(M) = 0$.
\end{lemma}

\section{Main Results}\label{sect-LCICase}

In this section, we prove the main results stated in the introduction. We begin with the hypersurface case, where the unipotent vanishing cycles module and its Hodge filtration capture the whole picture. We then turn to local complete intersections and introduce several singularity invariants, showing how they relate to each other and to the $\HRH$ invariant. The results beyond the hypersurface case can be viewed as different ways of generalizing this picture.

\subsection{Hypersurfaces}\label{sectionhypersurfaces} Let $Z$ be a hypersurface of an $n$-dimensional smooth variety $X$ defined by $f$. Let \[B_f = \bigoplus_{j\geq 0} \OX \partial_t^j\delta\] as defined in \textsection\ref{subsect-VFiltSpecSpec}, and \[F_{k-n}B_f = \bigoplus_{j\leq k}\OX \partial_t^j\delta.\] 

Recall that by \cite{DOR1}*{Thm. D(1)}, $Z$ is $k$-Hodge rational homology if \[F_{k-n}W_{n+1}\cohH^1_Z(\OX) = F_{k-n}\cohH^1_Z(\OX).\]
The local cohomology is captured by the unipotent nearby and vanishing cycles and their corresponding filtrations. More precisely, we have a short exact sequence \[0 \to {\rm Gr}_V^0(B_f) \xrightarrow[]{t \cdot} {\rm Gr}_V^1(B_f) \to \cohH^1_Z(\OX)\to 0,\] which is bi-strict with respect to the Hodge filtration and the weight filtration, where the weight filtration of the first two terms is induced by their monodromy operators, and underlies the sequence of mixed Hodge modules (see, for example, \cite{WHI}*{Thm. A}) \[0 \to \varphi_{f,1}\QQ_X^H[n] \xrightarrow[]{Var} \psi_{f,1}\QQ_X^H[n] (-1) \to \cohH^1_Z(\QQ^H_X[n]) \to 0.\]

Recall the convention for the Hodge filtration on nearby cycles, when indexing the Hodge filtration as for right $\Dmod$-modules, is as follows:
\[ F_p \psi_{f,1}(\shO_X) = F_{p-1} {\rm Gr}_V^1(B_f),\]
and that the weight filtration is defined as the monodromy filtration for the nilpotent operator ($N $ or $t\de_t$) centered at $n-1$. In general, the monodromy filtration satisfies (see \cite{SZMonodromy}*{Rmk. 2.3})
\[ W_{n-1+i}\Gr_V^1(B_f) = \sum_{\ell \geq \max\{0,-i\}} (t\de_t)^\ell \ker((t\de_t)^{1+i+2\ell}).\]

However, in our situation, $(t\de_t)^\ell \colon (\Gr_V^1(B_f),F) \to (\Gr_V^1(B_f),F[-\ell])$ is strict for all $\ell \geq 1$. It is an easy exercise, then, to see
\[ F_p W_{n-1+i}\Gr_V^1(B_f) = \sum_{\ell \geq \max\{0,-i\}} (t\de_t)^\ell F_{p-\ell} \ker((t\de_t)^{1+i+2\ell}).\]

This discussion, by taking $i = 0$, gives the containments
\[ F_p \ker(t\de_t) \subseteq F_p W_{n-1} \Gr_V^1(B_f) \subseteq F_p \ker(t\de_t) + t(F_p \Gr_V^0(B_f)), \]
and hence, the formula
\[ F_p W_{n+1} \cohH^1_Z(\shO_X) = \frac{F_p W_{n-1}{\rm Gr}_V^1(B_f) + t(F_p\Gr_V^0(B_f))}{t(F_p \Gr_V^0(B_f))} = \frac{{F_p} \ker(t\partial_t) + t(F_p {\rm Gr}_V^0(B_f))}{t(F_p{\rm Gr}_V^0(B_f))}.\]

\begin{proof}[Proof of \theoremref{thmhypersurfaces}]
Assume first that $F_{k-n}{\rm Gr}_V^0(B_f) = 0$. This means that every element in $F_{k-1-n}\Gr_V^1(B_f) = F_{k-1-n}(\psi_{f,1}(\shO_X)(-1))$ lies in the subset $\ker{ \partial_t} = \ker{t\partial_t} \subseteq W_{n-1} {\rm Gr}_V^1(B_f) = W_{n+1}(\psi_{f,1}(\QQ_X^H[n])(-1))$, where the containment follows by definition of the monodromy filtration centered at $n-1$. This implies that 
\[F_{k-1-n}\cohH^1_Z(\OX) \subseteq W_{n+1}\cohH^1_Z(\OX),\] and therefore, $\HRH(Z)\geq k-1$. \\

Suppose now that $F_{k-1-n}W_{n+1}\cohH^1_Z(\OX) = F_{k-1-n}\cohH^1_Z(\OX)$. By induction on $k$, we have that $F_{k-1-n} {\rm Gr}_V^0(B_f) = 0$. Thus, we have an isomorphism
\[ F_{k-1-n} {\rm Gr}_V^1(B_f) \cong F_{k-1-n} \cohH^1_Z(\shO_X).\]

To prove the claim, it suffices by strict surjectivity of $\de_t \colon ({\rm Gr}_V^1(B_f),F) \to ({\rm Gr}_V^0(B_f),F[-1])$ to prove that $F_{k-1-n} {\rm Gr}_V^1(B_f) \subseteq \ker(\de_t) = \ker(t\de_t)$. As $F_{k-1-n} {\rm Gr}_V^0(B_f) = 0$, we get 
\[ F_{k-1-n} \cohH^1_Z(\shO_X) = F_{k-1-n}\ker(N),\]
which finishes the proof.
\end{proof}

Suppose now that $Z$ has an isolated singularity. Let $F$ be the Milnor fiber of $Z$. It is well-known that \[\DR({\rm Gr}_V^0(B_f)) \cong H^{n-1}(F)_1\] supported on the singular point. Moreover, the dimension of the Hodge filtration of this cohomology is controlled by the spectral numbers. More precisely, if $\alpha_{f,i}$ are the spectral numbers, then \[\# \{ i \mid  \alpha_{f,i} = k\} = \dim\Gr_F^p H^{n-1}(F)_1\] for $p=n-k$ (and similarly for non-integer ones, see e.g., \cite{saito93}*{\textsection 3}). By using induction and the definition of the Hodge filtration and the de-Rham functor, the following result follows from \theoremref{thmhypersurfaces}.

\begin{corollary}\label{corollaryspectrum}
    Let $Z$ be a hypersurface of a smooth variety that has an isolated singularity at $x\in Z$. Then, \[\HRH(Z) =  \Spe_{\min,\ZZ}(Z,x) - 2.\]
\end{corollary}

\begin{remark}
    The results in this section can be proved with the techniques of the upcoming sections. The hypersurface case is a good illustration of objects introduced in the case of local complete intersections. See \remarkref{rmk-HypersurfaceCase}.
\end{remark}

\subsection{Reinterpretation using specialization} We now focus on the local complete intersection case. Let $Z = V(f_1,\dots, f_r) \subseteq X$ be a complete intersection subvariety of pure codimension $r$, where $X$ is a smooth irreducible variety of dimension $n$. We introduce in this section several integer invariants associated to $Z$ and show how they relate to each other. In the hypersurface case, we will see that these numbers are all essentially the same, except for the invariant defined using the Bernstein--Sato polynomial.

We have the short exact sequence of monodromic mixed Hodge modules on $X\times \mathbb A^r_z$,
\[ 0 \to L \to {\rm Sp}(B_f)^\ZZ \to Q^\ZZ \to 0,\]
which in this case encodes the morphism $i_* \psi_Z$. To see this, let $\sigma \colon X \to X\times \mathbb A^r_z$ be the inclusion of the zero section. By \propositionref{prop-Restriction} above, we see that $\sigma^!(M) = \sigma^!(M^\ZZ)$ for any monodromic module $M$.

We have by \cite{SaitoMHM}*{Pg. 269} that
\[ \sigma^!({\rm Sp}(B_f)) = \sigma^!(B_f) \cong i_* i^!\QQ_X^H[n],\]
where we also use $\sigma$ to denote the zero section in $X\times \mathbb A^r_t$ for the middle term, and the last equality follows by Base Change \cite{SaitoMHM}*{(4.4.3)}.

Applying $\sigma^!$ to the short exact sequence, we get an exact triangle
\[ \sigma^! L \to i_* i^! \QQ_X^H[n] \to \sigma^! Q^{\ZZ} \xrightarrow[]{+1}.\]

Recalling that $L = i_* \QQ_{Z\times \mathbb A^r_z}[n]$, it is a simple computation to check that \[\sigma^! L = i_* \QQ_Z^H[n-2r](-r) = (i_*\QQ_Z^H[d])[-r](-r),\] and so the morphism $\chi \colon \sigma^! L \to i_* i^! \QQ_X^H[n]$ is (up to non-zero scalar multiplication on the irreducible components of $Z$) the map $\psi_Z[-r](-r)$. By looking at the long exact sequence in cohomology, the only non-zero terms are the rightmost four:
\[ 0 \to \cohH^{r-1}(\sigma^! Q) \to (i_* \QQ_Z^H[d])(-r) \xrightarrow[]{\chi} \cohH^r_Z(\QQ_X^H[n]) \to \cohH^r(\sigma^! Q^{\ZZ}) \to 0.\]

\begin{proposition} \label{prop-PRSQ} In the notation above, we have the following:
\begin{enumerate} \item $F_{k-d} \gamma_Z$ is an isomorphism if and only if $F_{k-n} \cohH^{r-1}(\sigma^! Q^{\ZZ}) = 0$.

\item $F_{k-d}\gamma_Z^\vee$ is an isomorphism if and only if $\HRH(Z) \geq k$ if and only if $F_{k-n} \cohH^{r}(\sigma^! Q^{\ZZ}) = 0$
\end{enumerate}

In particular, $F_{k-n}\cohH^r(\sigma^! Q^\ZZ) = 0$ implies $F_{k-n} \cohH^{r-1}(\sigma^! Q^{\ZZ}) = 0$.
\end{proposition}
\begin{proof} Recall that $F_{k-d} \gamma_Z$ is an isomorphism if and only if it is injective. This is true if and only if $F_{k-d}\psi_Z$ is injective. As $\psi_Z(-r)$ and $\chi$ agree up to non-zero scalar multiples, this is equivalent to $F_{k-n}\chi$ being injective, which by the exact sequence is equivalent to the vanishing $F_{k-n}\cohH^{r-1}(\sigma^! \cQ^{\ZZ})$.

The other claim is shown similarly.    
\end{proof}

We can rephrase the results of the proposition purely in terms of the $V$-filtration.

\begin{proposition} \label{prop-PRSV} In the notation above, we have the following:
\begin{enumerate} \item The map $F_{k-d} \gamma_Z$ is an isomorphism if and only if \[\left(\sum_{i=1}^r t_i F_{k-n} {\rm Gr}_V^{r-1}(B_f)\right) \cap \left(\bigcap_{i=1}^r \ker(\de_{t_i} \colon F_{k-n} {\rm Gr}_V^r(B_f) \to F_{k-n+1} {\rm Gr}_V^{r-1}(B_f))\right) = 0.\]

\item The map $F_{k-d}\gamma_Z^\vee$ is an isomorphism if and only if $\HRH(Z) \geq k$ if and only if \[\left(\sum_{i=1}^r t_i F_{k-n} {\rm Gr}_V^{r-1}(B_f)\right) + \left(\bigcap_{i=1}^r \ker(\de_{t_i} \colon F_{k-n} {\rm Gr}_V^r(B_f) \to F_{k-n+1} {\rm Gr}_V^{r-1}(B_f))\right) = F_{k-n}{\rm Gr}_V^r(B_f).\]
\end{enumerate}

In particular $\HRH(Z) \geq k$ if and only if
\[ F_{k-n}{\rm Gr}_V^r(B_f) = \left(\sum_{i=1}^r t_i F_{k-n} {\rm Gr}_V^{r-1}(B_f)\right) \oplus \left(\bigcap_{i=1}^r \ker(\de_{t_i} \colon F_{k-n} {\rm Gr}_V^r(B_f) \to F_{k-n+1} {\rm Gr}_V^{r-1}(B_f))\right)\]
\end{proposition}
\begin{proof} These are immediate restatements of the conditions in the previous proposition, using the definition of the underlying filtered $\Dmod$-module of $Q$.
\end{proof}

Thus, we get a $V$-filtration characterization of rational smoothness in the local complete intersection case.

\begin{corollary} \label{cor-RSVFilt} In the above notation, $Z$ is a rational homology manifold if and only if 
    \[ {\rm Gr}_V^r(B_f) = \left(\sum_{i=1}^r t_i {\rm Gr}_V^{r-1}(B_f)\right) \oplus \left(\bigcap_{i=1}^r \ker(\de_{t_i} \colon {\rm Gr}_V^r(B_f) \to  {\rm Gr}_V^{r-1}(B_f))\right).\]
\end{corollary}

In fact, we can say the following:
\begin{theorem} Let $Z \subseteq X$ be a complete intersection defined by $f_1,\dots, f_r \in \shO_X(X)$. The following are equivalent:
\begin{enumerate}
\item ${\rm Sp}(B_f)^{\ZZ}$ is a pure Hodge module of weight $n$.

\item $N = 0$ on ${\rm Sp}(B_f)^{\ZZ}$.

\item $Q^{\ZZ}$ is a pure Hodge module of weight $n$.
\end{enumerate}

Moreover, any of those equivalent conditions implies the following (all of which are equivalent to each other):
\begin{enumerate} \item $Z$ is a rational homology manifold.

\item $Z \times \mathbb A^r$ is a rational homology manifold.

\item $L = \QQ^H_{Z\times \mathbb A^r}[n]$ is a pure Hodge module of weight $n$.
\end{enumerate}

If $r =1$, the converse holds.

\end{theorem}
\begin{proof} The equivalence of the second collection of three conditions is obvious.

The equivalence of the first and second conditions follows from the fact that the weight filtration on ${\rm Sp}(B_f)^\ZZ$ is the monodromy filtration for $N$ centered at $n$. Either condition implies the third, because $Q^{\ZZ}$ is a quotient of ${\rm Sp}(B_f)^{\ZZ}$ by definition.

To see that the third implies the first, recall that $W_n L = L$. Thus, for all $i\in \ZZ$, we have a short exact sequence,
\[ 0 \to {\rm Gr}^W_{n+i} L \to {\rm Gr}^W_{n+i} {\rm Sp}(B_f)^{\ZZ} \to {\rm Gr}^W_{n+i} Q^\ZZ \to 0,\]
and for $i > 0$, we have an isomorphism ${\rm Gr}^W_{n+i} {\rm Sp}(B_f)^\ZZ \cong {\rm Gr}^W_{n+i} Q^\ZZ$. Thus, $W_n Q^{\ZZ} = Q^{\ZZ}$ implies ${\rm Sp}(B_f)$ is pure.

We see that the condition that ${\rm Sp}(B_f)^\ZZ$ is pure of weight $n$ implies $L$ is pure of weight $n$. 

To prove the converse, assume $r =1$. Then $Z = \{f=0\}$ is a rational homology manifold if and only if $\varphi_{f,1}(\shO_X) = 0$ if and only if $N =0$ on $\psi_{f,1}(\shO_X)$. But then
\[ {\rm Sp}(B_f)^\ZZ = \bigoplus_{\ell \geq 1} {\rm Gr}_V^\ell (B_f)\]
clearly satisfies the vanishing $N =0$.
\end{proof}

\begin{remark} The difficulty in proving the converse of the above theorem in $r>1$ is the fact that we do not know if the inequality
\[ \min\{i \mid (s+r)^i {\rm Gr}_V^r(B_f) = 0\} \geq \min\{\ell \mid (s+r-1)^\ell {\rm Gr}_V^{r-1}(B_f) =0\}\]
is strict. This is true when $r=1$.

The inequality cannot be strict in general: indeed, if $Z$ is a rational homology manifold, then by \corollaryref{cor-RSVFilt}, we have
\[ {\rm Gr}_V^r(B_f) = \ker(s+r) + \sum_{i=1}^r t_i {\rm Gr}_V^{r-1}(B_f),\]
and so there are two possibilities: ${\rm Gr}_V^{r-1}(B_f) =0$, in which case $Q^\ZZ = 0$, or the inequality is an equality.
\end{remark}

\subsection{Integer invariants and relations among them}\label{sec-intinv}
We can now introduce the second integer invariant of this work, which is $p(Q^\ZZ,F)$. Immediately we obtain the following:
\begin{proposition} \label{prop-InequalityPRS} We have an inequality
\[ p(Q^\ZZ,F) +n -1 \leq \HRH(Z).\]
\end{proposition}
\begin{proof} By definition, $F_{p(Q^\ZZ,F) -1} Q^\ZZ = 0$, so
$F_{p(Q^\ZZ,F)-1} \cohH^r \sigma^!(\cQ) = 0$
by \propositionref{prop-Restriction}.
\end{proof}

Before introducing the next integer invariant, we study how the invariant $p(Q^\ZZ,F)$ relates to the invariants $p(\cQ^k,F) = \inf\{p \mid F_p \cQ^k \neq 0\}$.

\begin{lemma} \label{lem-structurePQ} The following hold:
\begin{enumerate}
    \item $p(Q^\ZZ,F) = p(\cQ^r,F) = p(\cQ^\ell,F) \text{ for all } \ell \geq r.$
    \item $p(\cQ^{1-\ell},F) = p(\cQ^1,F) + \ell$ for all $\ell \geq 0$.
    \item For all $k\in [1,r) \cap \ZZ$, we have
\[ p(\cQ^{k+1},F) \leq p(\cQ^k,F) \leq p(\cQ^{k+1},F)+ 1.\]
\end{enumerate}
\end{lemma}
\begin{proof} The filtered acyclicity of $B^j(\cQ,F)$ for $j > 0$, given by \propositionref{prop-KoszulFiltAcyclic} above, tells us that there are surjections
\[ \bigoplus_{i=1}^r F_p \cQ^{r+j-1} \xrightarrow[]{z_i} F_p \cQ^{r+j},\]
which inductively proves $p(\cQ^{\ell},F) \geq p(\cQ^r,F)$ for all $\ell \geq r$. The same filtered acyclicity tells us that the map
\[ F_p \cQ^j \xrightarrow[]{z_i} \bigoplus_{i=1}^r F_p \cQ^{j+1}\]
is injective for all $j > 0$, so that
\[ p(\cQ^{j},F) \geq p(\cQ^{j+1},F) \text{ for all } j > 0.\]

These together show that $p(\cQ^r,F) = p(\cQ^{\ell},F)$ for all $\ell \geq r$.

The filtered acyclicity of $C^j(\cQ,F)$ for all $j < 0$ gives surjections
\[ \bigoplus_{i=1}^r F_p \cQ^{j+1} \xrightarrow[]{\de_{z_i}} F_{p+1} \cQ^j,\]
which gives the inequality $p(\cQ^{j},F) \geq p(\cQ^{j+1},F)+1$. Moreover, it gives injections
\[ F_p \cQ^{j+r} \xrightarrow[]{\de_{z_i}} \bigoplus_{i=1}^r F_{p+1} \cQ^{j+r-1}\]
so that $p(\cQ^{j},F) \leq p(\cQ^{j+1},F)+1$ for all $j < r$.

This proves almost all claims, except we need to show that $p(\cQ^0,F) = p(\cQ^1,F)+1$.

The long exact sequence in cohomology for $\sigma^*$ applied to the short exact sequence \eqref{eq-QSES} gives isomorphisms
\[ \cohH^0 \sigma^*(Q) \cong \cohH^0 \sigma^*({\rm Sp}(B_f)) \cong \cohH^0 \sigma^*(B_f).\]
and all three modules vanish, because $B_f$ has no quotient object supported on $X \times \{0\}$.

Thus, $\cohH^0(C^0(\cQ,F)) = 0$, giving a surjection
\[ \bigoplus_{i=1}^r F_p \cQ^1 \xrightarrow[]{\de_{z_i}} F_{p+1} \cQ^0,\]
and proving the last remaining claim.
\end{proof}

This leads to a natural condition for equality in \propositionref{prop-InequalityPRS}.

\begin{proposition} \label{prop-equalityPRS} Assume $p(\cQ^{r-1},F) = p(Q^\ZZ,F)+1$. Then
\[ \HRH(Z) = p(Q^\ZZ,F) + n -1.\]
\end{proposition}
\begin{proof} The assumption tells us that $F_{p(Q^\ZZ,F)} \cQ^{r-1} = F_{p(Q^\ZZ,F)} {\rm Gr}_V^{r-1}(B_f) = 0$. Then
\[ F_{p(Q^\ZZ,F)} \cohH^r(\sigma^!(\cQ)) = F_{p(Q^\ZZ,F)}\cQ^r / \sum_{i=1}^r z_i F_{p(Q^\ZZ,F)} \cQ^{r-1} = F_{p(Q^\ZZ,F)} \cQ^r \neq 0,\]
proving that $\HRH(Z) < p(Q^\ZZ,F) + n$.
\end{proof}

The next integer invariant is the \emph{minimal integer spectral number} of $Z$ at a point $x\in Z$, which we denote by ${\rm Sp}_{\min,\ZZ}(Z,x)$.

An immediate application of \lemmaref{lem-lowerBoundSpectralNumbers} above gives the following lower bound. Here \[p(Q_x^\ZZ,F) = \min\{p\in \ZZ \mid (F_p \cQ^\ZZ)_x \neq 0\}\] is the lowest non-vanishing index of the stalk of the Hodge filtration at $x$.

\begin{proposition} \label{prop-SpectrumBound} Let $x\in Z$ be a point in the local complete intersection variety $Z$. Then
\[ {\rm Sp}_{\min,\ZZ}(Z,x) \geq p(Q^\ZZ_x,F) + n +1.\]
\end{proposition}

The same condition in \propositionref{prop-equalityPRS} above gives a condition which allows us to ensure equality in the proposition above in the isolated singularities case.

\begin{proposition} Let $x \in Z$ be an isolated singular point. Assume $p(\cQ^{r-1}_x,F) = p(Q^\ZZ_x,F)+1$. Then
\[ {\rm Sp}_{\min,\ZZ}(Z,x) = p(Q_x^\ZZ,F) + n +1.\]
\end{proposition}
\begin{proof} The assumption allows us to use \cite{DirksMicrolocal}*{Lem. 2.6}, which tells us that
\[ {\rm Supp}_{\mathbb A^r_z}(F_{p(\cQ^{\ZZ},F)} \cQ^\ZZ) = \mathbb A^r_z.\]
Then the proof goes through in exactly the same way as the last step of the proof of \cite{DirksMicrolocal}*{Thm. 1.1}.
\end{proof}

\begin{remark} We have the following: $Q^\ZZ = 0$ if and only if $p(Q^\ZZ,F) = + \infty$, which implies that ${\rm Sp}_{\min,\ZZ}(Z,x) = +\infty$ and that $Z$ is a rational homology manifold.

It is easy to check that $Q^\ZZ =0$ if and only if ${\rm Gr}_V^{r-1}(B_f) = 0$. \exampleref{eg-Torrelli} below shows, then, that $Q^\ZZ = 0$ is too strong: there is a complete intersection variety $Z$ which is a rational homology manifold but with $Q^\ZZ \neq 0$.
\end{remark}

\begin{remark} \label{rmk-HypersurfaceCase} As mentioned in the introduction, in the hypersurface case, the situation simplifies immensely. Let $Z = V(f) \subseteq X$. It is well known that $Z$ is a rational homology manifold if and only if $\varphi_{f,1}(\shO_X) = 0$, see for example, \cite{JKSY}*{Thm. 3.1}. By the previous remark, this is true if and only if $Q^\ZZ = 0$. In this remark, we give another proof of that fact.

We use the following notation as we have fixed a defining function for $Z$:
\[ {\rm Sp}_{\min,\ZZ}(Z,x) = {\rm Sp}_{\min,\ZZ}(f,x),\]
\[ \HRH(Z) = \HRH(f).\]

If $Q^\ZZ \neq 0$, then in this case we have $p({\rm Gr}_V^0(B_f),F) = p(Q^\ZZ,F) +1$. Indeed, this follows immediately from the fact that
\[ {\rm can} = \de_t \colon {\rm Gr}_V^1(B_f,F) \to {\rm Gr}_V^0(B_f,F[-1])\]
is strictly surjective. 

By \propositionref{prop-equalityPRS}, we see that $\HRH(f) = p(Q^\ZZ,F) + n -1 = p(\varphi_{f,1}(\shO_X),F) + n-2$, which shows that $\HRH(f) = +\infty$ if and only if $p(\varphi_{f,1}(\shO_X),F) = +\infty$, proving the equivalence mentioned at the beginning of the remark. Moreover, if $Z$ has an isolated singularity at $x\in Z$, then we get
\[ {\rm Sp}_{\min,\ZZ}(Z,x) = p(Q^\ZZ,F) + n +1 = p(\varphi_{f,1}(\shO_X),F) + n.\]

This gives another proof of the equality noted in \cite{JKSY}*{Formula (13)}.
\end{remark}

In \cite{CDMO}, when defining the minimal exponent for the complete intersection subvariety $Z = V(f_1,\dots, f_r) \subseteq X$, an auxiliary construction is used. This is the \emph{general linear combination hypersurface}, defined by $g = \sum_{i=1}^r y_i f_i$ on $Y = X\times \mathbb A^r_y$. Let $U = Y \setminus (X\times \{0\})$ with open embedding $j\colon U \to Y$ and let $\sigma \colon X\times \{0\} \to Y$ be the inclusion of the zero section. Then we have the exact triangle
\[ \sigma_* \sigma^! \varphi_{g,1}(\QQ^H_Y[n+r]) \to \varphi_{g,1}(\QQ^H_Y[n+r]) \to j_* \varphi_{g\vert_U,1}(\QQ^H_U[n+r]) \xrightarrow[]{+1}\]

In \cite{DirksMicrolocal}, the $\Dmod$-module $\varphi_{g,1}(\shO_Y)$ is compared to the Fourier-Laplace transform of ${\rm Sp}(B_f)$. This comparison implies that $\sigma_* \sigma^! \varphi_{g,1}(\QQ^H_Y[n+r])$ has a unique non-vanishing cohomology module in the local complete intersection case, namely, the 0-th one. This is due to the fact that $\sigma_* \sigma^*{\rm Sp}(B_f)$ has a unique non-vanishing cohomology module, and the Fourier-Laplace transform interchanges the two types of restriction to the zero section.

Thus, this exact triangle is actually a short exact sequence
\[ 0 \to \sigma_* \sigma^! \varphi_{g,1}(\QQ^H_Y[n+r]) \to \varphi_{g,1}(\QQ^H_Y[n+r]) \to j_* \varphi_{g\vert_U,1}(\QQ^H_U[n+r]) \to 0.\]

Each term in the short exact sequence is monodromic along the variables $y_1,\dots, y_r$. In fact, this is the Fourier-Laplace transform of the short exact sequence \eqref{eq-QSES}. As a corollary of this fact, we get the following.

\begin{corollary} \label{cor-QVanish} In the above notation, we see
\[ Q^\ZZ = 0 \text{ if and only if } \varphi_{g\vert_U,1}(\shO_U) = 0 \text{ if and only if } V(g\vert_U) \subseteq U \text{ is a rational homology manifold.}\]
In particular, $V(g\vert_U)$ being a rational homology manifold implies $Z$ is one, too.
\end{corollary}

As explained in the proof of \cite{DirksMicrolocal}*{Prop. 3.4}, there is an isomorphism
\[ F_p \varphi_{g,1}(\shO_Y)^{r-k} \cong F_{p-k+r} {\rm Gr}_V^{k}(B_f),\]
where the left-hand side is the corresponding monodromic piece. By the short exact sequence and strictness of morphisms with respect to the Hodge filtration, there is also an isomorphism
\begin{equation}\label{eq-phi}
    F_p j_* \varphi_{g\vert_U,1}(\shO_U)^{r-k} \cong F_{p-k+r}\cQ^k,
\end{equation} 
which we can use to prove the following.

\begin{proposition} Let $p(\varphi_{g\vert_U,1}(\shO_U),F) = \min\{p\mid F_p \varphi_{g\vert_U,1}(\shO_U) \neq 0\} = \min\{p\mid F_p j_*(\varphi_{g\vert_U,1}(\shO_U)) \neq 0\}$. Then
\[ p(\varphi_{g\vert_U,1}(\shO_U),F) = \min_{\ell \in [0,r-1]} \{p(\cQ^{r-\ell},F) -\ell\} = p(\cQ^1,F) - r+1.\]
\end{proposition}
\begin{proof} The last equality follows from \lemmaref{lem-structurePQ}, which implies
\[ p(\cQ^1,F) \leq p(\cQ^{\ell},F) + \ell -1 \text{ for all } \ell \in [1,r]\cap \ZZ,\]
and so $p(\cQ^1,F) - r +1 \leq p(\cQ^{r-\ell},F) - \ell$ for all $\ell \in [0,r-1]$.

Note that $p(j_*\varphi_{g\vert_U,1}(\shO_U),F) = p(\varphi_{g\vert_U,1}(\shO_U),F)$.

We have the formula \eqref{eq-phi} from above
\[ F_p j_* \varphi_{g\vert_U,1}(\shO_U)^{r-k} \cong F_{p-k+r}\cQ^k,\]
and so by definition, we get
\[ p(\varphi_{g\vert_U,1}(\shO_U),F) = \min_{k \in \ZZ} \{p(\cQ^{r-k},F) - k\}.\]

Thus, we need to check that this minimum is achieved for $k \in [0,r-1]\cap \ZZ$. For $k \geq r$, we have
\[ p(\cQ^{r-k},F) = p(\cQ^1,F) + 1+k-r,\]
and so
\[ p(\cQ^{r-k},F) - k = p(\cQ^{r-(r-1)},F) - (r-1),\]
and we get that $p(\varphi_{g\vert_U,1}(\shO_U),F) = \min_{k \in \ZZ_{< r}} p(\cQ^{r-k},F) - k$.

For $k \leq 0$, we have
\[ p(\cQ^{r-k},F) = p(\cQ^r,F),\]
and so
\[ p(\cQ^{r-k},F) - k  > p(\cQ^r,F) - r,\]
proving the desired equality.
\end{proof}

\begin{remark} \label{rmk-jumpsInterpretation} 
Another way to phrase the result is the following: let 
\[j(Z) = \# \{\ell \in [1,r-1]\cap \ZZ \mid p(\cQ^{r-\ell},F) = p(\cQ^{r-\ell+1},F)\}\] be the number of times when the lowest Hodge piece doesn't jump in the interval $[1,r-1]$. Then
\[ p(Q^\ZZ,F) = p(\varphi_{g\vert_U,1}(\shO_U),F) + j(Z).\]
Indeed, this follows from the observation
\[ p(\cQ^{r-\ell},F) = p(Q^\ZZ,F) + \ell - |\{ j \leq \ell \mid p(\cQ^{r-j},F) = p(\cQ^{r-j+1},F)\}|,\]
so that $p(\cQ^{r-\ell},F) - \ell = p(Q^\ZZ,F) - (\text{the number of missed jumps up to } \ell)$. The minimum of this is clearly when $\ell$ is maximal, proving the claim.
\end{remark}

\begin{corollary}\label{cor-HRHU} Let $j(Z)$ be the number of missed jumps in the interval $[1,r-1]$ as defined above. Then
\begin{enumerate}
    \item $\HRH(Z) \geq \HRH(g\vert_U) -r + j(Z)+1.$
    \item If $p(\cQ^{r-1},F) = p(Q^\ZZ,F)+1$ (which is true if $j(Z)=0$), then $0 \leq j(Z) < r-1$ and we have equality \[\HRH(Z) = \HRH(g\vert_U) -r+j(Z)+1.\]
\end{enumerate}
\end{corollary}
\begin{proof} We have equality
\[ p(Q^{\ZZ},F) = p(\varphi_{g\vert_U,1}(\shO_U),F) + j(Z),\]
and by \remarkref{rmk-HypersurfaceCase} the right-hand side is equal to
$\HRH(g\vert_U) - (n+r)+2 + j(Z)$.
Thus, we have
\[ \HRH(Z) \geq p(Q^\ZZ,F) +n -1 = \HRH(g\vert_U) -r + j(Z) +1,\]
as claimed.

For the last claim, equality holds by \propositionref{prop-equalityPRS}.
\end{proof}

We now describe the relation of the Bernstein--Sato polynomial to the $G$-filtration when $r=1$. Let $f \in \shO_X(X)$ define a non-empty hypersurface. The module $B_f = \bigoplus_{k\in \NN} \shO_X \de_t^k \delta_f$ admits an exhaustive filtration $G^\bullet B_f$ indexed by $\ZZ$ (defined below also for $r \geq 1$). This has the property that the Bernstein--Sato polynomial $b_f(s)$ of $f$ is the minimal polynomial of the action of $-\de_t t$ on ${\rm Gr}_G^0(B_f)$. As $f$ defines a non-empty hypersurface, it is easy to see that $(s+1) \mid b_f(s)$ using the usual description in terms of functional equations:
\[ b_f(s) f^s = P(s) f^{s+1} \text{ for some } P(s) \in \Dmod_X[s].\]

The reduced Bernstein--Sato polynomial $\widetilde{b}_f(s) = b_f(s)/(s+1)$ is an invariant of the singularities of $f$. The minimal exponent is
\[ \widetilde{\alpha}(f) = \min\{\lambda \in \QQ \mid \widetilde{b}_f(-\lambda) = 0\},\]
and so we can consider a similar formula, only considering integer roots:
\[ \widetilde{\alpha}_\ZZ(f) = \min\{j\in \ZZ \mid \widetilde{b}_f(-j) = 0\}.\]

We clearly have an inequality $\widetilde{\alpha}(f) \leq \widetilde{\alpha}_\ZZ(f)$.

The $G$-filtration has the property that
\[ {\rm Gr}_V^\lambda {\rm Gr}_G^0(B_f) \neq 0 \text{ if and only if } (s+\lambda) \mid b_f(s).\]

To study the reduced polynomial, Saito \cite{SaitoMicrolocal} introduced the \emph{microlocalization} $\widetilde{\cB}_f = B_f[\de_t^{-1}]$. This object carries a \emph{microlocal $V$-filtration} and a $G$-filtration. Importantly, these filtrations satisfy: the minimal polynomial of $-\de_t t$ on ${\rm Gr}_G^0(\widetilde{\cB}_f)$ is $\widetilde{b}_f(s)$ by \cite{SaitoMicrolocal}*{Prop. 0.3}, and
\[ {\rm Gr}_V^\lambda {\rm Gr}_G^0(\widetilde{\cB}_f) \neq 0\text{ if and only if } (s+\lambda) \mid \widetilde{b}_f(s).\]

Using the isomorphisms $\de_t^j\colon {\rm Gr}_V^{\lambda}{\rm Gr}_G^{k}(\widetilde{\cB}_f) \cong {\rm Gr}_V^{\lambda-j}{\rm Gr}_G^{k-j}(\widetilde{\cB}_f)$, we see then that
\[ \widetilde{\alpha}_\ZZ(f) \geq j \text{ if and only if } G^{-j+1} {\rm Gr}_V^0(\widetilde{\cB}_f) = 0.\]

Another important aspect of the microlocalization functor is that the natural map
\[ {\rm Gr}_V^0(B_f,F) \to {\rm Gr}_V^0(\widetilde{\cB}_f,F)\]
is an isomorphism, where $F_\bullet \widetilde{\cB}_f = \bigoplus_{k \leq \bullet+n} \shO_X \de_t^k \delta_f$. 

Using this property and the fact that $F_{k-n} \widetilde{\cB}_f \subseteq G^{-k}\widetilde{\cB}_f$, it follows that $p({\rm Gr}_V^0(B_f),F) \geq \widetilde{\alpha}_{\ZZ}(f)-n$.

By \propositionref{prop-SpectrumBound}, we get
\[ \widetilde{\alpha}_{\ZZ,x}(f) \leq {\rm Sp}_{\min,\ZZ}(f,x).\]

Even in the isolated hypersurface singularities case, strict inequality $\widetilde{\alpha}_{\ZZ}(f) < {\rm Sp}_{\min,\ZZ}(f,x)$ is possible \cite{JKSY}*{Rmk. 3.4d}.

\begin{remark} When $Z = V(f_1,\dots, f_r) \subseteq X$ is defined by a regular sequence, the minimal exponent is defined by
\[ \widetilde{\alpha}(Z) = \widetilde{\alpha}(g\vert_U),\]
where the right term is the minimal exponent for hypersurfaces defined above \cite{CDMO}.
Associated to the tuple $f_1,\dots, f_r$, there is also a Bernstein--Sato polynomial \cite{BMS}, related to the reduced Bernstein--Sato polynomial of $g$, by \cite{Mustata}:
\[ \widetilde{b}_g(s) = b_f(s).\]
Moreover, $(s+r) \mid b_f(s)$, so we can consider the \emph{reduced Bernstein--Sato polynomial} $\widetilde{b}_f(s) = b_f(s)/(s+r)$. The minimal exponent satisfies
$\widetilde{\alpha}(Z) = \min\{\lambda \in \QQ \mid \widetilde{b}_f(-\lambda) =0 \}$ as well by \cite{DirksMicrolocal}.
\end{remark}

One might ask whether $\widetilde{\alpha}_{\ZZ}(g\vert_U)$ satisfies a similar formula, that is, whether it is equal to 
\[\widetilde{\alpha}_{\ZZ}(Z) = \min\{j \in \ZZ \mid \widetilde{b}_f(-j) = 0\}.\] 

This equality would allow one to rewrite the results of \corollaryref{cor-HRHU} in terms of invariants of $Z$ rather than invariants of $g\vert_U$.

It is not true, however, that $\widetilde{b}_f(s) = \widetilde{b}_{g\vert_U}(s)$, and so the question is slightly subtle. In general, we have \cite{DirksMicrolocal}*{Lem. 5.2}:
\[ b_f(s) = \widetilde{b}_g(s) = \widetilde{b}_{g\vert_U}(s) \prod_{i\in I} (s+r+i),\]
where $I \subseteq \ZZ_{\geq 0}$ is some finite subset. So we can see that we have equality $\widetilde{\alpha}_\ZZ(Z) = \widetilde{\alpha}_\ZZ(g\vert_U)$ if either one is strictly less than $r$.

There are two cases: either $0\in I$ or $0 \notin I$. If $0\in I$, then we get
\[ \widetilde{b}_f(s) = \widetilde{b}_{g\vert_U}(s) \prod_{i\in I \setminus \{0\}} (s+r+i),\]
and so we have inequality $\widetilde{\alpha}_{\ZZ}(Z) \leq \widetilde{\alpha}_{\ZZ}(g\vert_U)$.

In the other case, we get
\[ \widetilde{b}_f(s) = \frac{\widetilde{b}_{g\vert_U}(s)}{(s+r)} \prod_{i\in I}(s+r+i),\]
and so we only get 
\[ \widetilde{\alpha}_{\ZZ}(g\vert_U) = \min\{r,\widetilde{\alpha}_{\ZZ}(Z)\}.\]

This case is possible, as \exampleref{eg-badbg} shows below.

\begin{corollary}\label{corprsbound} In the above notation, let $j(Z)$ be the number of missed jumps. Then
\begin{enumerate}
    \item $\HRH(Z) \geq \widetilde{\alpha}_{\ZZ}(g\vert_U) -r +j(Z)-1.$
    \item If $\widetilde{\alpha}_{\ZZ}(g\vert_U) > r$, then we have inequality \[\HRH(Z) \geq \widetilde{\alpha}_{\ZZ}(Z) -r +j(Z)-1.\]
\end{enumerate}
\end{corollary}
\begin{proof} We have
\[ \widetilde{\alpha}_{\ZZ}(g\vert_U) \leq p(\varphi_{g\vert_U,1}(\shO_U),F) + (n+r), \]
and so we get that
\[ p(Q^\ZZ,F) - j(Z) + (n+r) \geq \widetilde{\alpha}_\ZZ(g\vert_U),\]
giving $p(Q^\ZZ,F) + n -1 \geq \widetilde{\alpha}_{\ZZ}(g\vert_U) +j(Z) -r -1$. Then the claim follows from \propositionref{prop-InequalityPRS}.

If we assume $\widetilde{\alpha}_{\ZZ}(g\vert_U) > r$, this implies that $(s+r) \nmid \widetilde{b}_{g\vert_U}(s)$, and so we are in the case $0\in I$. Thus, we have inequality $\widetilde{\alpha}_{\ZZ}(Z) \leq \widetilde{\alpha}_{\ZZ}(g\vert_U)$.
\end{proof}

As a consequence of these results, we obtain \theoremref{thm-openU} and \corollaryref{cor-bs}.

\begin{proof}[Proof of \theoremref{thm-openU}.]
    The first statement is \corollaryref{cor-QVanish}. The first inequality follows from the proof of \corollaryref{corprsbound}, and the second inequality from the statement of the same corollary. Finally, the last inequality follows from \corollaryref{cor-HRHU}.
\end{proof}

\begin{proof}[Proof of \corollaryref{cor-bs}]
   The inequalities follow from \theoremref{thm-openU} and the last line of \corollaryref{corprsbound}.

   The implication $\widetilde{\alpha}_{\ZZ}(Z) = +\infty$ implies rational homology manifold follows from the inequality $\widetilde{\alpha}_{\ZZ}(Z) -r-1\leq \HRH(Z)$. In the hypersurface case, the converse follows from the observation that $\widetilde{\alpha}_{\ZZ}(Z) = + \infty$ if and only if for all $j\in \ZZ$, we have ${\rm Gr}_G^j {\rm Gr}_V^0(B_f) = 0$, but then by exhaustiveness of the filtration $G$, this is true if and only if ${\rm Gr}_V^0(B_f) =0$, which is equivalent to the rational homology manifold condition for hypersurfaces.
\end{proof}

\begin{remark} We have
\[ \HRH(Z) < + \infty \text{ if and only if } \HRH(Z) \leq \frac{d-3}{2}.\]
Combined with the inequality above, we see that if
\[ \widetilde{\alpha}_{\ZZ}(g\vert_U) > \frac{d-3}{2} + r -j(Z) +1 = \frac{n+r-1}{2} -j(Z),\]
then $Z$ is a rational homology manifold.
\end{remark}

\begin{example}\label{prsbound} The previous remark  is most clear when $r = 2$. Indeed, in that case, we have $j \in\{0,1\}$. If $j = 0$, then by \propositionref{prop-equalityPRS}, we see that $\HRH(Z) = p(Q^\ZZ,F) + n -1$.
Otherwise, if $j =1$, then we see that if $Z$ is not a rational homology manifold, we have the inequality
\[ \widetilde{\alpha}_{\ZZ}(g\vert_U) \leq \frac{n-1}{2}.\]
\end{example}

Now, we give some partial results on the invariant $\widetilde{\alpha}_\ZZ(Z)$. As in the hypersurface case above, the polynomial $b_f(s)$ is controlled by the $G$-filtration on $B_f$. This is defined as follows: the ring $\Dmod_T$ has a $\ZZ$-indexed filtration
\[ V^k \Dmod_T = \left\{ \sum_{\beta,\gamma} P_{\beta,\gamma} t^\beta \de_t^\gamma \mid P_{\beta,\gamma} \in \Dmod_X, |\beta| \geq |\gamma|+k\right\},\]
and so we can define an exhaustive filtration
\[ G^\bullet B_f = (V^\bullet \Dmod_T)\cdot \delta_f.\]

Then $b_f(s)$ is the minimal polynomial of the action of $s = -\sum_{i=1}^r \de_{t_i} t_i$ on ${\rm Gr}_G^0(B_f)$. Moreover, we have
\begin{equation} \label{eq-testRoot} {\rm Gr}_V^\lambda {\rm Gr}_G^0(B_f) \neq 0 \text{ if and only if } (s+\lambda) \mid b_f(s). \end{equation}

Thus, it is worthwhile to study the induced $G$-filtration on ${\rm Sp}(B_f)^\ZZ$. We define it as
\[ G^\bullet {\rm Sp}(B_f)^{\ZZ} = \bigoplus G^{\bullet+k} {\rm Gr}_V^{r+k}(B_f),\]
where the shift by $k$ is to make it so that $G^\bullet {\rm Sp}(B_f)^{\ZZ}$ is a filtration by sub-$\Dmod_T$-modules. This induces $G$-filtrations on $\cL$ and $\cQ$ by the short exact sequence \eqref{eq-QSES}. The $G$-filtration on $\cL$ is rather simple by \cite{DirksMicrolocal}*{Lem. 5.1}: it satisfies $G^0 \cL = \cL$. In particular, for all $j > 0$, we have an isomorphism
\[ {\rm Gr}_G^{-j}({\rm Sp}(B_f)) \cong {\rm Gr}_G^{-j}(\cQ).\]

This immediately leads to the following:

\begin{proposition}\label{prop-qneqz} Assume $(s+r+j) \mid b_f(s)$ for some $j > 0$. Then $Q^\ZZ \neq 0$.
\end{proposition}
\begin{proof} As $(s+r+j) \mid b_f(s)$, this means that ${\rm Gr}_G^0 {\rm Gr}_V^{r+j}(B_f) \neq 0$. This is the $r+j$th monodromic piece of ${\rm Gr}_G^{-j}({\rm Sp}(B_f))$, proving that
\[ {\rm Gr}_G^{-j}({\rm Sp}(B_f)^{\ZZ}) = {\rm Gr}_G^{-j}(\cQ^{\ZZ}) \neq 0,\]
and so $\cQ^\ZZ \neq 0$.
\end{proof}

\begin{remark} \label{rmk-WeirdBehavior} The result is a bit surprising, in view of the equality
\[ b_f(s) = \widetilde{b}_{g\vert_U}(s) \prod_{i\in I} (s+r+i).\]

Indeed, the claim says that if there is any factor $(s+r+j) \mid b_f(s)$ with $j > 0$, then there is also an integer root in $\widetilde{b}_{g\vert_U}(s)$, because $V(g\vert_U)$ is not a rational homology manifold in this case. So any non-zero element of $I$ implies the existence of an integer root in $\widetilde{b}_{g\vert_U}(s)$, and hence the existence of another integer root in $b_f(s)$.
\end{remark}

This remark naturally leads to the following conjecture:
\begin{conjecture} \label{conj-BFunction} For any regular sequence $f_1,\dots, f_r\in \shO_X(X)$, we have
\[ b_g(s) \mid b_{g\vert_U}(s)(s+r).\]

In other words, the set $I$ is either empty or equal to the singleton $\{0\}$.
\end{conjecture}

If $r = 1$, then we always have $b_g(s) = b_{g\vert_U}(s)(s+1)$ by \cite{LeeTSProduct}. For $r > 1$, the case $b_g(s) = b_{g\vert_U}(s)$ is possible, as we see in the following example.

\begin{example} \label{eg-badbg} There is a reduced, irreducible complete intersection variety $Z = V(f_1,f_2)$ with $b_g(s) = b_{g\vert_U}(s)$. Indeed, if we let $f_1 = x^2+y^3$ and $f_2 = xy+zw$, then Macaulay2 shows that
\[ b_g(s) = b_{g\vert_U}(s) =  (s+1)(s+2)^2 \left(s+\frac{5}{2}\right) \left(s+\frac{7}{3}\right) \left(s+\frac{8}{3}\right) \left(s+\frac{11}{6}\right)^2 \left(s+\frac{13}{6}\right)^2.\]
Note that, in this example, $Z$ is not normal.
\end{example}

\begin{remark} Note that if $Z$ has rational singularities, then $0\in I$. Indeed, \cite{CDMO}*{Cor. 1.7} tells us that $Z$ having rational singularities is equivalent to $\widetilde{\alpha}(Z) = \widetilde{\alpha}(g\vert_U) > r$. But if 
\begin{center}
    $b_g(s) = b_{g\vert_U}(s)\prod_{i\in I} (s+r+i)$ with $I \subseteq \ZZ_{>0}$,
\end{center} 
then we have
\[ (s+r) \mid b_f(s) = \widetilde{b}_g(s) = \widetilde{b}_{g\vert_U}(s) \prod_{i\in I} (s+r+i),\]
which forces $(s+r) \mid \widetilde{b}_{g\vert_U}(s)$ and so $\widetilde{\alpha}(g\vert_U) \leq r$, a contradiction.
\end{remark}

\begin{corollary} Assume $Z$ has rational singularities. Then
\[ \HRH(Z) \geq \widetilde{\alpha}_{\ZZ}(Z) -r +j(Z) -1,\]
where $j(Z)$ is the number of missed jumps.
\end{corollary}
\begin{proof} As $Z$ has rational singularities, we have $\widetilde{\alpha}(Z) = \widetilde{\alpha}(g\vert_U) > r$. Moreover, by the discussion above, $\widetilde{\alpha}_\ZZ(Z) = \widetilde{\alpha}_\ZZ(g\vert_U) \geq \widetilde{\alpha}(g\vert_U) > r$. So the result follows by \corollaryref{corprsbound}.
\end{proof}

In any case, we can collect our findings in the following corollary.

\begin{corollary} \label{cor-QMinExpZ} The following hold:
\begin{enumerate}
    \item If $(s+\ell) \mid b_f(s)$ for some integer $\ell \neq r$, then $Q^\ZZ \neq 0$.
    \item If $b_g(s) = b_{g\vert_U}(s)(s+r)$, then $\widetilde{\alpha}_{\ZZ}(Z) < +\infty$ if and only if $Q^\ZZ \neq 0$.
    \item If $b_g(s) = b_{g\vert_U}(s)$, then $Q^\ZZ \neq 0$.
\end{enumerate}
\end{corollary}
\begin{proof} For the first claim, if $\ell > r$, then this is the result of \propositionref{prop-qneqz}.

For $\ell < r$, use the equality
\[ b_f(s) = \widetilde{b}_{g\vert_U}(s) \prod_{i\in I} (s+r+i),\]
where $I \subseteq \ZZ_{\geq 0}$ implies that $(s+\ell) \mid \widetilde{b}_{g\vert_U}(s)$. Then we use \corollaryref{cor-QVanish} and \remarkref{rmk-HypersurfaceCase} to conclude.

For the second claim, we have $\widetilde{b}_f(s) = \widetilde{b}_g(s)/(s+r) = \widetilde{b}_{g\vert_U}(s)$, and so $\widetilde{\alpha}_\ZZ(Z) = \widetilde{\alpha}_{\ZZ}(g\vert_U)$. So the claim follows by \corollaryref{cor-QVanish} and \remarkref{rmk-HypersurfaceCase}.

For the last claim, we have $(s+r) \mid b_f(s) = \widetilde{b}_{g\vert_U}(s)$, so that $\widetilde{\alpha}_\ZZ(g\vert_U) < + \infty$ and again we use \corollaryref{cor-QVanish} and \remarkref{rmk-HypersurfaceCase}.
\end{proof}

\subsection{The case of isolated singularities}\label{sec-iso}
In this section, we discuss a deeper relation between the spectrum of an isolated local complete intersection singularity and the Hodge Rational Homology degree.

Recall that, in this case, we have the Milnor fiber, which can be defined in the following way. Let $(Z,x)$ be the germ of the isolated singularity, and let $\rho\colon (\mathcal{X},x)\to \Delta$ a smoothing with central fiber $Z$. We then let the Milnor fiber $F$ be the topological space $\mathcal{X}_t$, for $t\neq 0$, and note that its cohomology can be endowed with a mixed Hodge structure. The cohomology is nonzero only in degrees 0 and $d = \dim{Z}$, and is independent of the smoothing along with its Hodge filtration. For more details, see \cite{FLIsolated}*{\textsection 2.2 and \textsection 4}.

Following Steenbrink and the notation in \cite{FLIsolated}, let \[s_p = \dim \Gr_F^pH^d(F).\] These invariants are deeply connected to the link invariants and can be used to describe $\HRH(Z)$.

\begin{proposition}\label{propmilnor}
    Let $(Z,x)$ be an isolated local complete intersection singularity. Then $\HRH(Z)\geq k $ if and only if $s_{d-p}-s_p = 0$ for all $0\leq p\leq k$.
\end{proposition}

\begin{proof}
    We will use that \begin{equation}\label{eqMF}s_{d-p} - s_p = \ell^{p,d-p-1} - \ell^{p,d-p}\end{equation} \cite{FLIsolated}*{Prop. 2.11(ii)}, and \begin{equation}\label{eqMF2}\sum_{p=0}^k{s_{d-p}} \geq \sum_{p=0}^k{s_p}\end{equation} and if equality holds, then $\ell^{d-k-1,k} = \ell^{k+1,d-k-1}=0$ \cite{FLIsolated}*{Prop. 2.12}.

    Suppose $\HRH(Z)\geq k$. Then by \cite{DOR1}*{Cor. G}, $0=\ell^{d-p,p} = \ell^{p,d-p-1}$ for $p\leq k$, and by combining \eqref{eqMF}, \eqref{eqMF2}, and using induction on $k$ we get $s_{d-p}-s_p = 0$ for all $0\leq p\leq k$.

    Conversely, suppose $s_{d-p}-s_p = 0$ for all $0\leq p\leq k$. We proceed by induction. For $k=0$, \[0=s_d -s_0 = \ell^{0,d-1} - \ell^{0,d} = \ell^{d,0},\] since $\ell^{0,d}=0$. By \cite{DOR1}*{Cor. G}, we obtain the conclusion. Suppose now that we know the result up until degree $(k-1)$. Then \[s_{d-k} - s_k = \ell^{k,d-k-1} - \ell^{k,d-k} = \ell^{d-k,k} - \ell^{(k-1)+1, d-(k-1)-1}.\]  By the second part of \eqref{eqMF2}, we obtain that $\ell^{d-k,k} = 0$, and by \cite{DOR1}*{Cor. G} we conclude.
\end{proof}

We note that in this setting, the spectral numbers are identified with the cohomology of the Milnor fiber. Indeed, \[m_{\alpha, x} = \dim\Gr_F^p H^d(F)_{\lambda},\] where $p = \lfloor d+1 - \alpha \rfloor$, and $\lambda = \exp(-2\pi i\alpha)$, since the pullback to a point $\xi \in \{x\}\times \mathbb A^r$ corresponds to picking a nearby fiber of a 1-parameter smoothing of $Z$ (see \cite{DMS}*{Rmk. 1.3 (i)} for more details). Furthermore, spectral numbers partially recover the duality classically known for hypersurface singularities.

\begin{proposition}\label{propduality}
    Let $(Z,x)$ be an isolated local complete intersection singularity. Then, for $\alpha\notin \ZZ$, \[m_{\alpha,x} = m_{d+1-\alpha,x}.\]
\end{proposition}

\begin{proof}
    This is a consequence of duality applied to the module $\cQ$. We also show in the proof what the construction yields for $\alpha\in\ZZ$.
    
    Since we are working locally around the isolated singularity $x\in Z$,  we have that $\cQ$ is supported on $\{x\} \times \mathbb A^r$, hence $\cQ = i_{x*} \cN$ and so if $j_\xi\colon \{\xi\} \to \{x\}\times \mathbb A^r$, we have the identification
\[ i_\xi^* \cQ = j_\xi^* \cN,\]
and
\[ i_\xi^* \mathbf D \cQ = j_{\xi}^* (\mathbf D \cN) = \mathbf D j_{\xi}^! \cN = \mathbf D ( j_\xi^* \cN)(r)[2r],\]
the last equality following from the fact that $j_\xi$ is non-characteristic.

We get
\[ m_{\alpha,x}(\cQ) = \sum_{k\in \ZZ} (-1)^k \dim \Gr^F_{\lceil \alpha\rceil - d -1} \cohH^{k-r} i_\xi^*(\cQ^{\alpha+\ZZ})\]
\[ = \sum_{k\in \ZZ} (-1)^k \dim \Gr^F_{\lceil \alpha \rceil -d - 1} \cohH^{k-r}j_\xi^* (\cN^{\alpha+\ZZ}).\]

Note that for any mixed Hodge structure $H$, we have the relation
\[ \Gr^F_{\bullet} \mathbf D(H) = \mathbf D \Gr^F_{-\bullet}(H),\]
and in particular,
\[ \dim \Gr^F_{\bullet}\mathbf D(H) = \dim \Gr^F_{-\bullet}(H).\]

Hence, we have
\[ m_{\alpha,x}(\mathbf D\cQ) = \sum_{k\in \ZZ} (-1)^k \dim \Gr^F_{\lceil \alpha \rceil - d -1} \cohH^{k-r}i_\xi^*((\mathbf D \cQ)^{\alpha+\ZZ})\]
\[ = \sum_{k\in \ZZ} (-1)^k \dim \Gr^F_{\lceil \alpha \rceil - d -1} \cohH^{k-r} (\mathbf D(j_\xi^*(\cN^{-\alpha +\ZZ}))(r)[2r])\]
\[ = \sum_{k\in \ZZ} (-1)^k \dim \Gr^F_{\lceil \alpha \rceil - d -r -1} \cohH^{k+r}(\mathbf D(j_\xi^*(\cN^{-\alpha+\ZZ})))\]
\[ = \sum_{k\in \ZZ} (-1)^k \dim \mathbf D \Gr^F_{d + r+1 - \lceil \alpha \rceil} \cohH^{-k-r}(j_\xi^*(\cN^{-\alpha+\ZZ})).\]
\[ = \sum_{k\in \ZZ} (-1)^k \dim \Gr^F_{d +r+1-\lceil \alpha \rceil} \cohH^{k-r}(j_\xi^*(\cN^{-\alpha +\ZZ})).\]

Write $d + r +1 - \lceil \alpha \rceil = \lceil \mu \rceil - d -1$ where $\mu + \alpha \in \ZZ$. Thus,
\[ 2d +r + 2 - \lceil \alpha \rceil = \lceil \mu \rceil.\]

If $\alpha \in \ZZ$, then $\mu \in \ZZ$ and we get $\mu = d + n + 2 - \alpha$.

Otherwise, we write $\alpha = p - \varepsilon$ with $\varepsilon \in (0,1)$, then $\lceil \alpha \rceil = p$ and so we must have
\[ \mu = \lfloor \mu\rfloor + \varepsilon =  2 d + r +2 - p -1 + \varepsilon = 2 d +r + 1 - p + \varepsilon = d + n +1 - \alpha.\]

Hence, we get 
\[m_{\alpha,x}(\mathbf D \cQ) = \begin{cases} m_{d + n +2-\alpha,x}(\cQ) & \alpha \in \ZZ \\ m_{d + n + 1-\alpha,x}(\cQ) & \alpha \notin \ZZ\end{cases}.\]

Using the isomorphism $\cQ^{\neq \ZZ} \cong \Spe(B_f)^{\neq \ZZ}$ and $\mathbf D(\cQ^{\neq \ZZ}) \cong \Spe(B_f)^{\neq \ZZ}(n)$, we have
\[ \widehat{\Spe}(\cQ^{\neq \ZZ},x) = \widehat{\Spe}({\Spe}(B_f)^{\neq \ZZ},x) = t^{-n} \widehat{\Spe}({\Spe}(B_f)^{\neq \ZZ}(n),x) = t^{-n}\widehat{\Spe}(\mathbf D \cQ^{\neq \ZZ},x).\] The result follows.
\end{proof}

\begin{remark}
    The same duality might not hold for $m_{k,x}$, $k\in\ZZ$ \cite{DMS}*{Rmk. 1.3 (iv)}, and depends on the Milnor fiber of a 1-parameter smoothing of $Z$. Furthermore, to compare the integer spectrum numbers using the proof above, we would need to consider the spectral numbers of $\mathcal{L}$ and its dual.
\end{remark}

\begin{proof}[Proof of \theoremref{thmspectrum}]
    We show that if $\Spe_{\min,\ZZ}(Z,x)\geq k+2$, then $\HRH(Z)\geq k$ around $x$. By \propositionref{propmilnor}, it is enough to verify $s_{d-p}-s_p = 0$ for all $p\leq k$. It is immediate to see that \[s_p = \sum_{\alpha\in (d-p, d-p+1]} m_{\alpha,x}.\] Therefore, \[s_{d-p}-s_p = \sum_{\beta\in (p, p+1]} m_{\beta,x} - \sum_{\alpha\in (d-p, d-p+1]} m_{\alpha,x} = m_{p+1,x} - m_{d-p+1, x},\] where the last equality follows from \propositionref{propduality}.

    Suppose $\Spe_{\min,\ZZ}(Z,x)\geq k+2$, that is, $m_{1,x} = \cdots = m_{k+1, x} = 0$. Then by \eqref{eqMF2} we have \[ 0\leq \sum_{p=0}^k{s_{d-p}} - \sum_{p=0}^k{s_p} = - (m_{d+1,x} + \cdots + m_{d-k+1, x}), \] and thus, $m_{d+1,x} = \cdots = m_{d-k+1, x} = 0$. Therefore, $s_{d-p} = s_p$ for all $p\leq k$.
\end{proof}

\section{Examples}\label{sec-ex}

\subsection{Thom-Sebastiani examples} Many of the singularity invariants are easier to control when using defining equations in separate collections of variables (of ``Thom-Sebastiani'' type), as we see now.

We will make use of the product formula for Verdier specializations, see \cite{DMS}*{Section 3} for details. Given $Z_i \subseteq X_i$, we consider the subvariety $Z_1\times Z_2 \subseteq X_1\times X_2$. Let $Z_1 = V(f_1,\dots, f_r)$ and let $Z_2 = V(g_1,\dots, g_\rho)$, so that $Z_1\times Z_2 = V(f_1,\dots, f_r, g_1,\dots, g_\rho)$. 

We consider modules $B_f, \cB_g$ and $\cB_{(f,g)}$. Then \cite{DMS}*{Prop. 3.2} gives an isomorphism
\[ {\rm Sp}(\cB_{(f,g)},F) = {\rm Sp}(B_f,F) \boxtimes {\rm Sp}(\cB_g,F).\]

In particular, there are isomorphisms
\begin{equation} \label{eq-ProductFormula} {\rm Gr}_V^\alpha(\cB_{(f,g)},F) = \bigoplus_{\alpha_1 + \alpha_2 = \alpha} {\rm Gr}_V^{\alpha_1}(B_f,F) \boxtimes {\rm Gr}_V^{\alpha_2}(\cB_g,F).\end{equation}

\begin{example}[An example of Torrelli] \label{eg-Torrelli} In \cite{Torrelli}, it is noted that there exists a complete intersection variety $Z$ which is a rational homology manifold but such that $\widetilde{\alpha}_{\ZZ}(Z) < +\infty$. The example follows from the observation that if $f = x^2+y^2+z^2$ and $g = u^2 + v^2 +w^2$, then we have \[b_f(s) = b_g(s) = (s+1)\left(s+ \frac{3}{2}\right).\] Hence, $V(f)$ and $V(g)$ are both rational homology manifolds, and their product $Z = V(f)\times V(g) = V(f,g) \subseteq \mathbb A^6$ is also a rational homology manifold.

The Thom-Sebastiani rule for the roots of the Bernstein--Sato polynomial \cite{BMS}*{Thm. 5} gives
\[ b_{(f,g)}(s) = (s+2)\left(s+\frac{5}{2}\right)(s+3),\]
and so $\widetilde{b}_{(f,g)}(s)$ has an integer root. In this example, $\widetilde{\alpha}_\ZZ(Z) = 3$. As $3 \neq 2$, we see by \remarkref{rmk-WeirdBehavior} above that this implies $\widetilde{\alpha}_{\ZZ}(g\vert_U) < +\infty$ (it is not hard to check that it is also equal to $3$ in this case). Thus, $Q^\ZZ \neq 0$. Moreover, note that since $\widetilde{\alpha}_{\ZZ}(g\vert_U) = 3$, then \corollaryref{corprsbound} recovers the fact that $Z$ is a rational homology manifold.

In fact, we can be rather explicit using the Product Formula \eqref{eq-ProductFormula} above. First of all, as $f,g$ are homogeneous with an isolated singular point, their $V$-filtrations are easy to compute (see \cite{SaitoOnHodgeFilt}*{(4.2.1)}). As $V(f),V(g)$ are rational homology manifolds, we have that ${\rm Gr}_V^j(B_f) = {\rm Gr}_V^j(\cB_g) = 0$ for all $j \leq 0$. Thus, we have
\[ {\rm Gr}_V^\lambda(B_f) \neq 0 \implies \lambda \in (\frac{1}{2} \ZZ \setminus \ZZ) \cup \ZZ_{\geq 1},\]
and similarly for ${\rm Gr}_V^{\lambda}(\cB_g)$.

Thus, we have
\[ {\rm Gr}_V^1(\cB_{(f,g)}) = \bigoplus_{\alpha \in \frac{1}{2} \ZZ \setminus \ZZ} {\rm Gr}_V^{\alpha}(B_f) \boxtimes {\rm Gr}_V^{1-\alpha}(\cB_g),\]
where we know we cannot include any $\alpha \in \ZZ$ in the direct sum because such an $\alpha$ would have to satisfy $\alpha > 0$ and $1-\alpha > 0$ in order to be non-zero.

By similar reasoning, we have
\[{\rm Gr}_V^2(\cB_{(f,g)})= \bigoplus_{\alpha \in \frac{1}{2}\ZZ} {\rm Gr}_V^{\alpha}(B_f) \boxtimes {\rm Gr}_V^{2-\alpha}(\cB_g)\]
     \[= ({\rm Gr}_V^1(B_f)\boxtimes {\rm Gr}_V^1(\cB_g)) \oplus \bigoplus_{\alpha \in \frac{1}{2} \ZZ \setminus \ZZ} {\rm Gr}_V^{\alpha}(B_f) \boxtimes {\rm Gr}_V^{2-\alpha}(\cB_g).
    \]

The claim is that this is the decomposition in \corollaryref{cor-RSVFilt} above. It is not hard to see that \[{\rm Gr}_V^1(B_f)\boxtimes {\rm Gr}_V^1(\cB_g) = \bigcap_{i=1}^2 \ker\left(\de_{t_i} \colon {\rm Gr}_V^2(\cB_{(f,g)}) \to {\rm Gr}_V^1(\cB_{(f,g)})\right),\] using that ${\rm Gr}_V^0$ vanishes for both $B_f$ and $\cB_g$.

Finally, for $\alpha \in \frac{1}{2} \ZZ \setminus \ZZ$, we have
\[ t_1 {\rm Gr}_V^{\alpha-1}(B_f) = {\rm Gr}_V^\alpha(B_f),\]
and similarly for $\cB_g$, showing that
\[ \bigoplus_{\alpha \in \frac{1}{2}\ZZ \setminus \ZZ} {\rm Gr}_V^\alpha(B_f) \boxtimes {\rm Gr}_V^{2-\alpha}(\cB_g) = t_1 {\rm Gr}_V^1(\cB_{(f,g)}) + t_2 {\rm Gr}_V^{1}(\cB_{(f,g)}).\]

Recall from \remarkref{rmk-jumpsInterpretation}, we say that the $\ell$-th jump is missed for $\ell \in [1,r-1]\cap \ZZ$ if  $p(\cQ^{r-\ell},F) = p(\cQ^{r-\ell+1},F)$. Note that in the above example, the first jump is missed: $p(Q^\ZZ,F) = p(\cQ^{r-1},F)$.
\end{example}

\begin{example}[Failure of Thom-Sebastiani type rules] An example related to the above shows that the property of being a (partial) rational homology manifold is not well-behaved under Thom-Sebastiani sums of hypersurfaces. 

Indeed, if $f = x_1^2 + \dots + x_n^2$ and $g = y_1^2 + \dots + y_m^2$ with $n,m$ both odd, then $V(f), V(g)$ are rational homology manifolds. However, their sum $f+g$ has Bernstein--Sato polynomial \[b_{f+g}(s) = (s+1)\left(s+ \frac{n+m}{2}\right),\] and hence it has an extra integer root. In this way, the Thom-Sebastiani sum of two rational homology manifold hypersurfaces need not remain a rational homology manifold.
\end{example}

We can use our constructions to obtain equality in \cite{DOR1}*{Thm. C}. When $Z$ is LCI with ${\rm HRH}(Z) \geq 0$, this inequality takes the form
\begin{equation} \label{eq-ppBound} 2{\rm HRH}(Z) + 3 \leq {\rm codim}_Z (Z_{\rm nRS}),\end{equation}
where the right-hand side is the codimension of the non-rational homology manifold locus of $Z$.

Indeed, if $Z$ is a local complete intersection $(2k+3)$-fold with $k$-rational singularities and which is not a $\QQ$-homology manifold, then $Z$ gives an example of equality in \eqref{eq-ppBound}. Indeed, in this case, ${\HRH}(Z) = k$, as if $\HRH(Z) \geq k+1$, this would imply $Z$ is a rational homology manifold by the inequality. Thus, the non-rational homology manifold locus is isolated.

Here are two more examples of this phenomenon:

\begin{example}\label{ex-eqg1} Let $f = x_1^2 + \dots + x_{2m}^2$ for $m\in \ZZ_{\geq 2}$ and $Z = \{f=0\} \subseteq \mathbb A^{2m}$. Then $b_f(s) = (s+1)(s+m)$, and we have
\[ \widetilde{\alpha}_\ZZ(f) = m = {\rm Sp}_{\min,\ZZ}(f) = \HRH(Z) +2.\]
 As $Z$ is a hypersurface, we have
\[ {\rm lcdef}(Z) = {\rm lcdef}_{\rm gen}(Z) = 0.\]
Thus, \eqref{eq-ppBound} gives
\[ 0 + 2(m-2) + 3 \leq {\rm codim}_Z(Z_{\rm nRS}).\]
But the right-hand side is $\dim Z = 2m-1$, as $Z$ has isolated singularities and is not a rational homology manifold. So again, we see that equality holds in \eqref{eq-ppBound}.
\end{example}

\begin{example}\label{ex-eqg2}
Equality for \eqref{eq-ppBound} also holds in the example of \cite{JKSY}*{Rmk. 3.4d}. Indeed, there it is shown that for $h = x^6 + y^5 + x^3y^3 + z^5+w^3$, we have ${\rm Sp}_{\min,\ZZ}(h) = 2$, and so \[\HRH(\{h=0\}) = 0.\] Equality holds because $\{h=0\}$ has dimension $3$.
\end{example}

However, strict inequality is also possible in \eqref{eq-ppBound}:

\begin{example}[An example of strict inequality]\label{ex-neqg} Let $f = \sum_{i=1}^{2m} x_i^{m}$ for $m>2$. This defines a hypersurface $Z$ of dimension $2m-1$ with an isolated singularity. Moreover, its minimal exponent is
\[\widetilde{\alpha}(f) = \sum_{i=1}^{2m} \frac{1}{m} = 2,\]
so that $Z$ is not a rational homology manifold. As $f$ is weighted homogeneous, we have $\widetilde{\alpha}(f) = {\rm Sp}_{\min}(f) = {\rm Sp}_{\min,\ZZ}(f)$, which gives
\[ \HRH(Z) = {\rm Sp}_{\min,\ZZ}(Z) - 2 = 0.\]
Thus, $Z_{\rm nRS}$ has codimension $2m-1$ in $Z$, which is strictly larger than ${\rm lcdef}_{\rm gen}(Z) + 2\HRH(Z) + 3 = 3$.
\end{example}

\subsection{Examples with liminal singularities} Recall that a local complete intersection $Z$ is said to have $\ell$-liminal singularities for some non-negative integer $\ell$ if its singularities are $\ell$-Du Bois but not $\ell$-rational (this is equivalent to $\widetilde{\alpha}(Z) = r+\ell$). Below we study the jumps when $Z$ has $l$-liminal singularities, or more generally when $\widetilde{\alpha}(Z)$ is an arbitrary integer.

\begin{example}\label{ex-jumpsliminal} The jumps are rather explicit when $\widetilde{\alpha}(Z) = r+\ell \in \ZZ$.

For $\ell < 0$, this condition is equivalent to $\delta_f \in V^{r+\ell}B_f \setminus V^{>r+\ell}B_f$. In this case, \[\widetilde{\alpha}(Z) = {\rm lct}(X,Z) = r+\ell\] (where ${\rm lct}(-)$ stands for {\it log-canonical threshold}) and we have
\[ p(\cQ^{r+\ell},F) = p(\cQ^{r+\ell+1},F) = \dots = p(Q^\ZZ,F) = -n.\]

The claim is that $p(\cQ^{r+\ell - k},F) = k - n$ for all $k\geq 0$. Indeed, as $\delta_f \in V^{r+\ell}B_f$, we see that $F_{k-n} B_f \subseteq V^{r+\ell-k}B_f$ using that $F_{k-n}B_f = \bigoplus_{|\alpha|\leq k} \shO_X \de_t^\alpha \delta_f$. As a result, $F_{(k-1)-n} B_f \subseteq V^{>r+\ell-k}B_f$, and so we get the desired claim. Thus, in this case, the first $-\ell$ jumps are missed, and after that, every jump is hit.

For $\ell \geq 0$, this condition is equivalent to $F_{\ell - n} B_f \subseteq V^rB_f$ and $F_{\ell+1-n}B_f \not \subseteq V^{>r-1}B_f$. The last condition implies that $F_{\ell-n} \cQ^r \neq 0$, and in fact, it is easy to see that in this case, $\ell -n = p(Q^\ZZ,F)$. Moreover, by the same reasoning as above, we see that
\[ p(\cQ^{r-k},F) = p(Q^\ZZ,F) + k \text{ for all } k\geq 0,\]
proving that, in this example, no jumps are missed. 

In the case $\widetilde{\alpha}(Z) = r+\ell \geq r$, the variety $Z$ is $\ell$-Du Bois but not $\ell$-rational. As the difference between these two is measured by the property that $\HRH(Z) \geq \ell$, we know that $\HRH(Z) \leq \ell -1$. By \propositionref{prop-InequalityPRS}, using that $p(Q^\ZZ,F) = \ell -n$, we see that we actually have equality $\HRH(Z) = \ell-1$, which is guaranteed in this case also by \propositionref{prop-equalityPRS}, once we know that the first jump is not missed.
\end{example}

\bibliography{bib}

\end{document}